\documentclass[12pt]{article}

\usepackage{geometry,amsmath,mathtools,blkarray,amssymb,amsthm,setspace,xcolor,bbm,tikz-cd,xparse,amsfonts,authblk,graphics,datetime,thmtools,comment,bm}
\usepackage[en-US]{datetime2}
\usepackage[normalem]{ulem}
\usepackage{caption,subcaption}
\usepackage{multirow}
\usepackage{booktabs}
\usepackage[mathlines]{lineno}
\usepackage{graphicx}

\usepackage{tikz,tikz-3dplot}
\usetikzlibrary{decorations.pathreplacing}
\usetikzlibrary{positioning}

\usepackage[shortlabels]{enumitem}
\usepackage[colorlinks,linkcolor=blue,citecolor=blue,urlcolor=blue,bookmarks=false,hypertexnames=true]{hyperref}
\usepackage[backend=biber, style=bwl-FU, sortcites=false, maxcitenames=2, mincitenames=1, maxbibnames=8, uniquelist=false]{biblatex}

\DeclareDelimFormat{postnotedelim}{\addcomma\space}
\let\cite\textcite

\allowdisplaybreaks
\bibliography{ref.bib}

\makeatletter
\@ifundefined{newcounteralias}{}{%
  \renewcommand\thmt@autorefsetup{%
    \@xa\def\csname\thmt@envname autorefname\@xa\endcsname
      \@xa{\thmt@thmname}%
  }%
}
\makeatother

\theoremstyle{plain}
\newtheorem{theorem}{Theorem}[section]
\newtheorem{proposition}[theorem]{Proposition}
\newtheorem{lemma}[theorem]{Lemma}
\newtheorem{corollary}[theorem]{Corollary}

\newtheorem{claim}[theorem]{Claim}

\theoremstyle{definition}

\newtheorem{definition}{Definition}[section]
\newtheorem{example}{Example}
\newtheorem{assumption}{Assumption}[section]

\newtheorem{remark}{Remark}[section]

\renewcommand\thmcontinues[1]{Continued}

\newcommand{\norm}[1]{\left\|#1\right\|}
\newcommand{\ip}[2]{\left<#1, #2\right>}
\newcommand{\argmin}{\mathop{\mathrm{argmin}}}

\renewcommand{\div}{\mathrm{div}}
\newcommand{\id}{\mathrm{id}}

\newcommand{\indic}{\bm{1}}

\NewDocumentCommand{\defmathletter}{m}{%
    \expandafter\newcommand\csname b#1\endcsname{\mathbb{#1}}%
    \expandafter\newcommand\csname c#1\endcsname{\mathcal{#1}}%
}
\NewDocumentCommand{\defmathletters}{>{\SplitList{,}}m}{\ProcessList{#1}{\defmathletter}}
\defmathletters{A,B,C,D,E,F,G,H,I,J,K,L,M,N,O,P,Q,R,S,T,U,V,X,Y,Z}

\NewDocumentCommand{\defvector}{m}{%
    \expandafter\newcommand\csname v#1\endcsname{\mathbf{#1}}%
}
\NewDocumentCommand{\defvectors}{>{\SplitList{,}}m}{\ProcessList{#1}{\defvector}}
\defvectors{A,B,C,D,E,F,G,H,I,J,K,L,M,N,O,P,Q,R,S,T,U,V,X,Y,Z}
\defvectors{a,b,c,d,e,f,g,h,i,j,k,l,m,n,o,p,q,r,s,t,u,v,x,y,z}

\title{Central Limit Theorems for Optimal Transport Functionals}
\author{Haruki Kono\thanks{Email: h.kono@r.hit-u.ac.jp}}
\affil{Hitotsubashi University}
\date{October 12, 2026}
\begin{document}
\maketitle

\begin{abstract}
We establish central limit theorems for empirical plug-in estimators of scalar functionals of quadratic optimal transport maps and potentials, including Wasserstein costs and several risk measures constructed from transport-based multivariate quantiles. 
We study the one-sample problem with a known reference distribution and i.i.d. observations from the target distribution. 
A natural approach is to combine a limit theorem for the empirical transport map with the functional delta method. We show, however, that in dimensions at least two, empirical transport-map errors admit no nonzero weak limit in $L^2$ under any normalizing scale.
Instead, we linearize the functional directly in the target distribution, obtaining an influence-function expansion with a remainder controlled by the transport-map error.
The influence function is characterized by a linear elliptic partial differential equation. 
Under regularity conditions, we use this expansion to prove central limit theorems in low dimensions and establish validity of the empirical bootstrap. 
We also derive semiparametric efficiency bounds and show that the plug-in estimators attain them under the same conditions.
\end{abstract}

\section{Introduction} \label{sec:introduction}

\subsection{Motivation and basic setup}
Optimal transport studies how to move probability mass from a source distribution to a target distribution at minimal cost.
For quadratic cost, one minimizes the average squared Euclidean distance between source points and their matched destinations, and the minimum is the squared Wasserstein distance.
If the source is absolutely continuous and both distributions have finite second moments, the theorem of \cite{brenier1991polar} gives a unique optimal map, up to source-null sets, that is the gradient of a convex function; see also \cite{villani2003topics, villani2009optimal}.

Optimal transport maps also provide a multivariate analogue of quantile functions.
In one dimension, the quadratic transport from the uniform distribution on $(0, 1)$ to a distribution with distribution function $F$ is its quantile function $F^{- 1}.$
In higher dimensions, a transport map from a chosen reference distribution assigns an outcome vector to each reference point, so that reference points play the role of multivariate ranks.
This construction underlies the transport-based quantiles of \cite{chernozhukov2017monge} and the center-outward quantiles of \cite{hallin2021distribution}.
These maps define scalar risk measures, including maximum correlation and its tail-restricted versions; see \cite{ekeland2012comonotonic, beirlant2020center}.
Statistical inference for such functionals must account for uncertainty in the estimated transport map.

We study this inferential problem when the reference distribution is known and the outcome distribution is estimated from independent observations.
To formulate the parameters, let $\mu$ denote the known absolutely continuous reference law and $P$ the unknown outcome law, with supports $\overline \Omega$ and $\overline \Lambda$ for bounded convex domains $\Omega, \Lambda \subset \bR^d.$
Write $T_P \coloneqq \nabla \varphi_P$ for the quadratic transport map and its convex potential, so that $T_P (Z)$ has law $P$ when $Z$ has law $\mu.$
We remove the additive ambiguity in the potential by imposing $\mu \varphi_P = 0,$ where $\mu f \coloneqq \int f  d \mu.$
For a real-valued functional $\cJ$ of the normalized potential, the parameter of interest is
\begin{equation} \label{eq:introtarget}
    \theta (P)
    \coloneqq
    \cJ (\varphi_P)
    .
\end{equation}
As we see below, we allow for functionals that depend not only on the value of $\varphi_P$ but also on its gradients.

A basic example is a linear integral of the transport map against a given bounded vector weight $w : \Omega \to \bR^d$:
$$
    \cJ_w (\varphi_P)
    \coloneqq
    \int_\Omega w (z)^\top \nabla \varphi_P (z)  d \mu (z).
$$
Squared Wasserstein distance is a nonlinear evaluation of the map:
$$
    \cJ_{\mathrm W} (\varphi_P)
    \coloneqq
    \int
        \norm{z - \nabla \varphi_P (z)}^2
    d \mu (z)
    =
    W_2^2 (\mu, P)
    .
$$
Maximum tail correlation is a tail-weighted linear evaluation introduced by \cite{beirlant2020center} using the spherical-uniform reference law.
We consider the uniform-ball version of \cite[Section 4.1.2]{sadhu2024stability}.
Let $\mu$ be uniform on the unit ball $\Omega$ and, for $\tau \in (0, 1),$ set $E_\tau \coloneqq \{z \in \Omega {\mid} \norm{z} > (1 - \tau)^{1/d}\},$ so that $\mu (E_\tau) = \tau.$
Then
$$
    \kappa_\tau (P)
    \coloneqq
    \frac{1}{\tau}
    \int_{E_\tau} z^\top \nabla \varphi_P (z)  d \mu (z).
$$
This measure averages the inner product of reference ranks and their assigned outcomes over the outer $\tau$ fraction of ranks.

Given independent observations $Y_1, \ldots, Y_n$ with law $P,$ let $P_n \coloneqq n^{- 1} \sum_{i = 1}^n \delta_{Y_i}$ be the empirical target law.
Write $T_{P_n} = \nabla \varphi_{P_n}$ for the transport from $\mu$ to $P_n,$ with $\mu \varphi_{P_n} = 0.$
The empirical plug-in estimator is $\widehat \theta_n \coloneqq \cJ (\varphi_{P_n}).$
A natural approach to its limiting distribution would first establish a function-space limit for $\varphi_{P_n} - \varphi_P$ and then apply the functional delta method to $\cJ.$
Indeed, \cite[Theorem 2.4 and Corollary 2.11]{gonzalez2026empirical} establish a central limit theorem for empirical potentials in an $L^q$ topology in dimensions two and three, and obtain Gaussian limits for weighted averages of the potential.
The functional delta method likewise applies to nonlinear summaries that are differentiable in the relevant $L^q$ norm.

This argument does not automatically extend to functionals involving gradients. 
Applying the functional delta method requires a derivative that is continuous in the topology of the potential limit. 
For gradient-dependent functionals, the first variation can contain boundary or interface terms that are not controlled by the $L^q$ norm.
The potential central limit theorem in $L^q$ therefore does not by itself justify the delta method for such a functional.

A natural stronger space is $H^1 (\Omega),$ whose norm controls both potential values and square-integrable weak gradients on $\Omega.$
However, even this route is unavailable for the exact empirical transport.
\begin{theorem}[Failure of process convergence]\footnote{Related failures of function-space convergence have been established in other transport models.
\cite[Proposition 4]{sadhu2024stability} rule out nonzero weak limits in $L^2$ for empirical transport-map errors when the source is continuous and the population target is supported on a prescribed finite set, under their regularity assumptions.
\cite[Theorem 7]{manole2024central} obtain a corresponding obstruction for maps estimated from kernel-smoothed densities on the flat torus in dimensions $d \geq 3,$ under their bandwidth conditions.
Theorem \ref{thm:processobstruction} concerns the unsmoothed empirical target of a smooth population law and also includes dimension two.
}
\label{thm:processobstruction}
Suppose Assumption \ref{ass:base} holds and $P_n \coloneqq n^{- 1} \sum_{i = 1}^n \delta_{Y_i}$ for independent observations with law $P.$
Normalize the potentials by $\mu \varphi_{P_n} = \mu \varphi_P = 0.$
In every dimension $d \geq 2,$ if $a_n \to \infty$ is deterministic and
$a_n (\varphi_{P_n} - \varphi_P)$ is tight in $H^1 (\Omega),$ then
$$
    a_n = o \left(\sqrt n\right)
    \qquad \text{and} \qquad
    a_n \norm{\varphi_{P_n} - \varphi_P}_{H^1 (\Omega)}
    \longrightarrow 0
    \quad \text{in probability}.
$$
The same statement holds for $a_n (T_{P_n} - T_P)$ in $L^2 (\mu; \bR^d).$
In particular, the root-$n$ errors are not tight, and no diverging deterministic normalization has a nonzero weak limit in these spaces.
\end{theorem}
The proof is given in Appendix \ref{app:processproofs}.

Fortunately, the absence of a function-space limit does not preclude Gaussian limits for selected scalar evaluations.
We therefore study the composite parameter $P \mapsto \cJ (\varphi_P)$ directly and bound the approximation error for each functional, without requiring convergence of the full potential in $H^1 (\Omega).$

\subsection{Overview of the results}
Our main result approximates the change in $\theta$ by the change in the expectation of a single function of the target outcome.
For another law $R$ on $\overline \Lambda,$ write $T_R = \nabla \varphi_R$ for its transport map, normalize $\mu \varphi_R = 0,$ and set $\theta (R) = \cJ (\varphi_R).$
Under the functional conditions below, we construct $\psi_P$ with $P \psi_P \coloneqq \int \psi_P  d P = 0$ and obtain the following expansion, with remainder $\cR (R, P),$ exponent $r \in (1, 2],$ and a constant $C_P$ independent of $R$:
\begin{equation} \label{eq:introexp}
    \begin{aligned}
        \theta (R)
        -
        \theta (P)
        & =
        (R - P) \psi_P
        +
        \cR (R, P)
        ,
        \\
        | \cR (R, P) |
        & \leq
        C_P \norm{T_R - T_P}_{L^2 (\mu)}^r
        .
    \end{aligned}
\end{equation}
Here $(R - P) \psi_P \coloneqq \int \psi_P d R - \int \psi_P d P.$
At $R = P_n,$ the leading term is the centered sample average $n^{- 1} \sum_{i = 1}^n \psi_P (Y_i).$
The function $\psi_P$ thus determines its limiting variance once the remainder is negligible at the root-$n$ scale.

The construction of $\psi_P$ starts by linearly approximating the change in $\cJ$ caused by a change in the potential.
We then represent this linear approximation, up to a controlled error, by $(R - P) \psi_P.$
Finding the representing function amounts to solving a linear elliptic partial differential equation determined by the population transport and the derivative of $\cJ.$
The proof verifies the regularity of this solution and bounds both the error in the initial functional approximation and the error in its representation as a change in expectation.

The remainder exponent $r$ depends on the functional.
Under the stated smoothness conditions, integrals of smooth functions of potential values and gradients, including squared Wasserstein distance, have a quadratic remainder ($r = 2$).
For maximum tail correlation, the weight jumps across the inner boundary of $E_\tau.$
We decompose the tail integral into the integral over the full reference ball minus that over the inner ball.
Controlling the contribution near this inner boundary gives $r = 4 / 3.$
The same bound applies to subgroup evaluations over regions with smooth boundaries whose closures lie in $\Omega.$

Dimension enters through the rate at which the empirical map approaches $T_P.$
The expansion yields root-$n$ Gaussian limits when $\sqrt n \norm{T_{P_n} - T_P}_{L^2 (\mu)}^r$ tends to zero in probability.
The available transport rates verify this condition in dimensions two and three for $r = 2,$ and in dimension two for $r = 4 / 3.$
The expansion itself holds in every dimension $d \geq 2;$ in higher dimensions its bound need not make the remainder negligible at the root-$n$ scale.

The expansion (\ref{eq:introexp}) also applies after resampling the observations with replacement and recomputing the transport statistic.
This gives empirical bootstrap confidence intervals in the same dimensions.
We also establish the information bound for the parameter: no estimator regular under local perturbations of the target law can have asymptotic variance below $P \psi_P^2.$
The empirical estimator attains this bound whenever its remainder is negligible at the root-$n$ scale.

\subsection{Related literature}
For general treatments of optimal transport, Wasserstein geometry, and Brenier maps, see \cite{villani2003topics, villani2009optimal}.
In statistics, an optimal transport map from a reference law to $P$ can be interpreted as a multivariate quantile function.
\cite{chernozhukov2017monge} develop Monge--Kantorovich quantiles, ranks, signs, and depth from this perspective, together with empirical counterparts and consistency results.
Thus the map $T_P$ studied here also has an established interpretation as a transport-based quantile map.
\cite{hallin2021distribution} develop the closely related center-outward distribution and quantile functions using a spherical-uniform reference distribution.
\cite{ghosal2022multivariate} establish consistency and rates for empirical transport ranks and quantiles, with applications to testing.
We study scalar evaluations of the quantile map and its potential under a smooth reference law.

Smooth perturbations of population distributions lead to differential equations describing changes in transport maps.
\cite{loeper2005regularity} studies such changes through the linearized Monge--Amp\`ere equation, and \cite{gigli2011holder} establishes quantitative continuity of optimal maps along curves of measures.
On smooth uniformly convex domains, \cite[Theorem 3.2]{gonzalez2025linearization} differentiate transport along smooth density paths and characterize the derivative by a linear elliptic boundary-value problem.
\cite[Theorems 1.1, 1.4, and 1.5]{caja2026stability} develop stability estimates, a divergence-form description of the linear response, and a second-variation formula for squared Wasserstein distance.
These results describe population perturbations; applying them to an empirical target requires control of discrete perturbations as well.

Quantitative linearization also provides approximations that can be applied to discrete targets.
\cite{goldman2021quantitative} develop harmonic approximations of optimal transport plans for rough measures.
\cite[Proposition 4.5]{gonzalez2026influence} establish asymptotic linearity for integrals of empirical map errors against a class of weights constructed from smooth functions, and obtain root-$n$ Gaussian limits in dimensions two and three.
Their equation (73) bounds the difference between such an integral and a centered empirical average by a power of the map error.
For a general functional, we identify when its derivative admits this type of weighted-map representation and control the additional nonlinear functional remainder.
Subgroup evaluations further require control near the subgroup boundary.
Section \ref{sec:testedmap} gives the precise correspondence with their result.

The minimax estimation of smooth optimal transport maps is studied by \cite{hutter2021minimax}, who construct estimators attaining near-minimax rates, and \cite{manole2024plugin}, who establish minimax rates for plug-in estimators.

\cite{manole2024central} develop limit distribution theory for smooth optimal transport maps on the flat torus.
For $d \geq 3,$ their Theorem 5 gives pointwise Gaussian limits for maps from a known source to a kernel density estimator of the target, with normalization $\sqrt{n h_n^{d - 2}},$ where $h_n \to 0$ is the bandwidth, under smoothness and bandwidth conditions that make the bias negligible.
Their quantitative linearization of the Monge--Amp\`ere equation (Theorem 2) reduces the analysis to a linear elliptic equation with random forcing and periodic boundary conditions.
These results concern pointwise estimation of a smooth map on the torus; our results concern root-$n$ inference for scalar functionals of empirical maps on bounded convex domains.

For transport with a fixed smoothing kernel, \cite[Section 4.2]{gonzalez2025linearization} establish a central limit theorem for maps to the smoothed target law, while \cite{goldfeld2024smooth,goldfeld2024inference} obtain limit distributions and bootstrap consistency for smoothed Wasserstein distances.
For entropic transport with a fixed positive regularization parameter, \cite{delbarrio2023entropic} prove a central limit theorem for transportation costs, \cite{gonzalez2024weak} derive limits for potentials and couplings, and \cite{goldfeld2024limit} establish limits for maps, potentials, and Sinkhorn divergences.

Function-space limits for potentials give another route to scalar inference, provided that the functional is continuous and differentiable in the topology of the limit.
\cite[Theorem 2.4 and Section 2.4]{gonzalez2026empirical} prove a low-dimensional $L^q$ central limit theorem and bootstrap consistency for empirical potentials, modulo additive constants.
These results apply to suitable summaries of potential values.
Functionals defined through gradients need not extend continuously to $L^q,$ because this norm does not control derivatives; their limits therefore require a separate argument.
Section \ref{sec:limitcomparison} compares the two approaches.

Semidiscrete optimal transport concerns a continuous source and a population target supported on a finite set of known locations, with unknown masses.
\cite{sadhu2024stability} establish stability of these maps, together with Gaussian limits and bootstrap consistency for scalar functionals.
Their Proposition 4, discussed above, shows that scalar Gaussian limits can coexist with failure of $L^2$ process convergence.
In our setting the population target instead has a smooth density, and the support of its empirical approximation consists of the observed sample points and grows with the sample size.
Thus their finite-support population theory does not directly cover the estimator studied here.

Inference for transportation costs is an established special case of scalar transport inference.
\cite{del2019central} prove expectation-centered central limit theorems for empirical quadratic transportation costs in general dimension; replacing the expectation of the estimator by the population cost requires additional bias control.
\cite{hundrieser2024unifying} obtain distributional limits through directional differentiability under empirical-process conditions.
\cite{manole2024plugin} use smooth-base stability to establish population-centered cost inference and efficiency, including empirical estimators in low dimensions.
Our functional class also includes summaries of potential values and subgroup evaluations of the transport map.
Section \ref{sec:limitcomparison} relates the Wasserstein specialization of our expansion to these cost results.
\subsection{Organization}
Section \ref{sec:setup} defines the model and functional assumptions.
Section \ref{sec:expansion} establishes the expansion and verifies the functional examples.
Section \ref{sec:statistics} derives empirical limits and bootstrap confidence intervals, and Section \ref{sec:efficiency} establishes the information bounds.
The appendix contains all omitted proofs.
\subsection{Notation}
For a probability law $Q,$ write $Q f \coloneqq \int f  d Q.$
The spaces $L^q (Q)$ and $L_0^q (Q)$ consist of functions with finite $q$th moment and their mean-zero subspace, respectively, with $\norm{f}_{L^q (Q)} \coloneqq (Q |f|^q)^{1 / q}$ for $1 \leq q < \infty.$
Vector norms are Euclidean, matrix norms are induced operator norms, $\top$ denotes transpose, and $I_d$ is the identity matrix.
For symmetric matrices, $A \preceq B$ means $x^\top A x \leq x^\top B x$ for every $x.$
The continuous dual of a normed space $E$ is $E^\ast,$ with $\norm{L}_{E^\ast} \coloneqq \sup_{\norm{v}_E \leq 1} |L (v)|.$

We write $\nabla,$ $D,$ and $D^2$ for gradients, Jacobians, and Hessians.
For a domain $U,$ $C^{k, \gamma} (\overline U)$ denotes functions with continuous derivatives through order $k$ on $\overline U$ and $\gamma$-H\"older continuous derivatives of order $k,$ where $0 < \gamma \leq 1.$
The case $\gamma = 1$ means Lipschitz continuity.
We write $C^\infty (\overline U)$ for restrictions of smooth functions on a neighborhood of $\overline U,$ and $C_c^\infty (U)$ for smooth functions compactly supported in $U.$
For locally integrable functions $v$ and $w,$ $w$ is the weak derivative $\partial_j v$ if $\int_U v  \partial_j \xi = -\int_U w \xi$ for every $\xi \in C_c^\infty (U).$
The Sobolev space $H^k (U)$ consists of functions with square-integrable weak derivatives through order $k,$ and $H_0^1 (U)$ is the closure of $C_c^\infty (U)$ in $H^1 (U).$
We use $\norm{\cdot}_\infty$ for the supremum or essential supremum norm, as appropriate, and $\operatorname{Lip} (f)$ for the Lipschitz constant.

To define boundary regularity, write $x = (x_{-d}, x_d)$ with $x_{-d} = (x_1, \ldots, x_{d - 1})^\top,$ and let $B_r$ be the open ball of radius $r$ centered at zero in the Euclidean space determined by context.
A domain $U$ has $C^{k, \gamma}$ boundary, for an integer $k \geq 1$ and $0 < \gamma \leq 1,$ if each $x_0 \in \partial U$ admits a rigid coordinate change $\Phi$ with $\Phi (x_0) = 0,$ a radius $r > 0,$ and $h \in C^{k, \gamma} (\overline{B_r})$ such that $h (0) = 0,$ $\norm{h}_\infty < r,$ and, with $C_r \coloneqq B_r \times (-r, r),$
\begin{align*}
    \Phi (U) \cap C_r 
    = 
    \{x \in C_r \mid x_d > h (x_{-d})\}
    , 
    \qquad
    \Phi (\partial U) \cap C_r 
    = 
    \{x \in C_r \mid x_d = h (x_{-d})\}
    .
\end{align*}
Replacing $C^{k, \gamma}$ by $C^k$ defines a $C^k$ boundary.
A function on such a boundary or hypersurface has class $C^{k, \gamma}$ if its expression in each graph chart has that regularity.

The symbols $\bP$ and $\bE$ denote probability and expectation under the sampling law.
We use $\rightsquigarrow$ for weak convergence and $\xrightarrow{\bP}$ for convergence in probability.
The notation $O_P (c_n)$ means boundedness in probability after division by $c_n,$ and $o_P (c_n)$ means convergence to zero in probability after that division.

\section{Setup, definitions, and assumptions} \label{sec:setup}

\subsection{Distributions and transport}

Let $d \geq 2$ throughout.
Let $\Omega, \Lambda \subset \bR^d$ be bounded convex domains.
The known source $\mu$ is absolutely continuous with support $\overline \Omega,$ and the unknown target $P$ has support $\overline \Lambda.$
For a measurable map $S,$ the push-forward $S_\# Q$ is the law of $S (X)$ when $X \sim Q.$
For every probability law $R$ on $\overline \Lambda,$ absolute continuity of $\mu$ ensures existence and $\mu$-almost-everywhere uniqueness of the quadratic optimal transport map by Brenier's theorem (\cite{brenier1991polar}; \cite[Theorem 2.12]{villani2003topics}).
We denote this map by
\begin{equation} \label{eq:transport}
    T_R
    \coloneqq
    \underset{S_\# \mu = R}{\argmin}
    \int_\Omega
        \norm{z - S (z)}^2
    d \mu (z)
    .
\end{equation}
For probability measures $Q, R$ with finite second moments, the squared Wasserstein distance is
$$
    W_2^2 (Q, R)
    \coloneqq
    \inf_{\pi \in \Pi (Q, R)}
    \int
        \norm{x - y}^2
    d \pi (x, y)
    ,
$$
where $\Pi (Q, R)$ denotes the probability measures with marginals $Q$ and $R.$

The following definition collects the support and density conditions used to obtain smooth population transport.
We use the class of regular distributions in \cite[Definition 2.2]{gonzalez2026influence}, stated here with open domains and closed supports.
\begin{definition}[Regular distribution] \label{def:regular}
For the exponent $\alpha \in (0, 1),$ a probability measure $Q$ on $\bR^d$ is regular if its support is $\overline U$ for a bounded convex domain $U$ with $C^{2, 1}$ boundary and it has a density $q \in C^{1, \alpha} (\overline U)$ satisfying
$$
    0
    <
    \inf_{\overline U} q
    \leq
    \sup_{\overline U} q
    <
    \infty
    .
$$
\end{definition}
We impose these conditions on the source and population target.
\begin{assumption}[Population distributions] \label{ass:base}
For $d \geq 2$ and exponent $0 < \alpha < 1,$ the measures $\mu$ and $P$ are regular in the sense of Definition \ref{def:regular}.
Their supports are $\overline \Omega$ and $\overline \Lambda,$ and their densities are denoted by $\rho$ and $p,$ respectively.
\end{assumption}

Assumption \ref{ass:base} includes the volume-uniform distribution on a smooth convex body, but excludes the spherical-uniform distribution used for center-outward quantiles (\cite{hallin2021distribution}).

Since $P$ is also absolutely continuous, its Brenier map to $\mu,$ denoted by $F_P,$ is the almost-everywhere inverse of $T_P.$
The following lemma records the regularity theorem stated in \cite[Theorem 2.3]{gonzalez2026influence}, based on \cite{caffarelli1992boundary, caffarelli1996boundary, chen2021global}, together with the inverse and change-of-variables identities for these regular maps.
\begin{lemma}[Population transport regularity] \label{lem:mapregularity}
Under Assumption \ref{ass:base}, the maps $T_P$ and $F_P$ have representatives in $C^{2, \alpha} (\overline \Omega; \bR^d)$ and $C^{2, \alpha} (\overline \Lambda; \bR^d),$ respectively.
They are mutually inverse on the closed domains and are diffeomorphisms between $\Omega$ and $\Lambda.$
Define
\begin{equation} \label{eq:H}
    H_P (z)
    \coloneqq
    DT_P (z)
    .
\end{equation}
The Jacobian $H_P$ is symmetric, and there are constants $C_P < \infty$ and $0 < m_P \leq M_P < \infty$ such that
\begin{gather*}
    \norm{T_P}_{C^{2, \alpha} (\overline \Omega)}
    +
    \norm{F_P}_{C^{2, \alpha} (\overline \Lambda)}
    \leq
    C_P
    \qquad
    \text{and}
    \qquad
    m_P I_d
    \preceq
    H_P (z)
    \preceq
    M_P I_d
    ,
    \qquad
    z \in \overline \Omega
    .
\end{gather*}
The matrix fields $H_P$ and $H_P^{- 1}$ belong to $C^{1, \alpha} (\overline \Omega),$ and
$$
    p (T_P (z)) \det H_P (z)
    =
    \rho (z)
    ,
    \qquad
    z \in \overline \Omega
    .
$$
\end{lemma}

\subsection{Potentials and the functional domain} \label{sec:potentialspace}
Define the zero-mean space
\begin{equation} \label{eq:V}
    \cV
    \coloneqq
    \{v \in H^1 (\Omega) \mid \mu v = 0\}
    ,
    \qquad
    \norm{v}_\cV
    \coloneqq
    \norm{\nabla v}_{L^2 (\mu)}
    .
\end{equation}
Under Assumption \ref{ass:base}, this is a Hilbert space with inner product
$\ip{u}{v}_\cV \coloneqq \int_\Omega \nabla u^\top \nabla v d \mu.$\footnote{Let $\bar v \coloneqq | \Omega |^{- 1} \int_\Omega v d z$ be the Lebesgue mean.
Since $\mu v = 0,$ the density bounds and Poincar\'e's inequality give $\norm{v}_{L^2 (\mu)} \leq 2 \norm{v - \bar v}_{L^2 (\mu)} \leq C \norm{\nabla v}_{L^2 (\mu)}$ (\cite{payne1960optimal}).
Thus $\norm{\cdot}_\cV$ is equivalent to the $H^1$ norm on $\cV.$
Since $v \mapsto \mu v$ is continuous on $H^1 (\Omega),$ its kernel $\cV$ is closed and hence complete.}

Brenier's theorem also characterizes the transport map: for every probability law $R$ on $\overline \Lambda,$ the map $T_R$ is $\mu$-almost everywhere the gradient of a convex function on $\Omega$ (\cite[Theorem 2.12]{villani2003topics}).
We call such a function a transport potential.
A transport potential is determined only up to an additive constant.
We choose the normalization $\mu \varphi_R = 0;$ this is one possible convention for selecting a representative.
For each probability law $R$ on $\overline \Lambda,$ let $\varphi_R$ denote the convex potential satisfying
$$
    \varphi_R \in \cV
    ,
    \qquad
    \nabla \varphi_R
    =
    T_R
    \quad
    \mu \text{-almost everywhere}
    .
$$
We use its continuous representative on $\overline \Omega.$
In particular, $H_P = D^2 \varphi_P.$
The next lemma establishes that this normalization is well-defined and records the bounds and continuity used below.
\begin{lemma}[Bounds and continuity of potentials] \label{lem:potentials}
Under Assumption \ref{ass:base}, put $L_\Lambda \coloneqq \sup_{y \in \overline \Lambda} \norm{y}$ and $M \coloneqq L_\Lambda \operatorname{diam} (\Omega).$
For every probability law $R$ on $\overline \Lambda,$ there is a unique normalized convex potential $\varphi_R \in \cV,$ with a continuous representative satisfying
$$
    \operatorname{Lip} (\varphi_R)
    \leq
    L_\Lambda
    ,
    \qquad
    \norm{\varphi_R}_\infty
    \leq
    M
    .
$$
The map $R \mapsto \varphi_R$ is continuous from weak convergence on $\overline \Lambda$ to the $\cV$ norm.
For any two probability laws $R, S$ on $\overline \Lambda,$
\begin{equation} \label{eq:normid}
    \norm{\varphi_R - \varphi_S}_\cV
    =
    \norm{T_R - T_S}_{L^2 (\mu)}
    .
\end{equation}
\end{lemma}
The proof is in Appendix \ref{app:regularity}.

The Kantorovich dual constructs the transport potential by optimizing over a pair of functions.
For a probability law $R$ on $\overline \Lambda,$ put $K_R \coloneqq \operatorname{supp} R.$
Since the marginal second moments do not depend on the coupling, minimizing quadratic cost is equivalent to maximizing $\int z^\top y  d \pi (z, y).$
Kantorovich duality gives
\begin{equation} \label{eq:transportdual}
    \sup_{\pi \in \Pi (\mu, R)} \int z^\top y  d \pi (z, y)
    = \inf\limits_{\substack{
        a \in C (\overline \Omega),\ b \in C (K_R) \\
        \forall z, y \ 
        a (z) + b (y) \geq z^\top y
    }} \left(\int_\Omega a  d \mu + \int_{K_R} b  d R\right),
\end{equation}
where $C (K)$ denotes the continuous real-valued functions on $K.$
Both sides are linear programs, and both extrema are attained; see \cite[Chapter 1]{villani2003topics}.

For a function $f$ on a compact set $K \subset \bR^d,$ its convex conjugate is $f^\ast (x) \coloneqq \sup_{y \in K} (x^\top y - f (y)).$
For a given $b,$ the smallest feasible choice of $a$ is therefore
\begin{equation} \label{eq:dualpotential}
    b^\ast (z) = \sup_{y \in K_R} (z^\top y - b (y)), \qquad z \in \overline \Omega.
\end{equation}
Eliminating $a$ gives the semi-dual problem
$$
    \inf\limits_{b \in C (K_R)} \left(\int_\Omega b^\ast  d \mu + \int_{K_R} b  d R\right).
$$
Its objective is convex in $b$ and generally nonlinear.
When $R$ has finite support, $b$ is a finite vector; this reduction underlies the methods in Remark \ref{rem:computation}.
For a minimizer $b_R,$ the gradient of $b_R^\ast$ transports $\mu$ to $R$ (\cite[Theorem 2.12]{villani2003topics}), so
$$
    \varphi_R = b_R^\ast - \int_\Omega b_R^\ast  d \mu.
$$
\par
The domain of the functional is the set of normalized transport potentials,
$$
    \cF
    \coloneqq
    \{\varphi_R \mid R \text{ is a probability law on } \overline \Lambda\}
    \subset
    \cV
    .
$$
For a functional $\cJ : \cF \to \bR,$ Borel measurable in the metric inherited from $\cV,$ define
\begin{equation} \label{eq:target}
    \theta (R)
    \coloneqq
    \cJ (\varphi_R)
    .
\end{equation}
Lemma \ref{lem:potentials} ensures that evaluating $\cJ$ at an empirical potential gives a measurable statistic.

\subsection{Functional assumptions}
We maintain Assumption \ref{ass:base} throughout this subsection.
Suppose a continuous linear form $D_P \in \cV^{\ast}$ approximates the change in the functional:
$$
    \cJ (\varphi_R)
    -
    \cJ (\varphi_P)
    \approx
    D_P (\varphi_R - \varphi_P)
    .
$$
To obtain a statistical limit, we seek a function $\psi_P$ for which this first variation is approximated by $(R - P) \psi_P.$
At $R = P_n,$ the latter is a centered sample average.
If $\psi_P$ is smooth on $\overline \Lambda$ and $P \psi_P = 0,$ the push-forward identities and a first-order Taylor expansion give
\begin{align*}
    (R - P) \psi_P
    &=
    \int_\Omega
        (\psi_P (T_R) - \psi_P (T_P))
    d \mu
    \\
    &\approx
    \int_\Omega
        \nabla \psi_P (T_P)^\top \nabla (\varphi_R - \varphi_P)
    d \mu
    \\
    &=
    \int_\Omega
        \nabla (\psi_P \circ T_P)^\top H_P^{- 1} \nabla (\varphi_R - \varphi_P)
    d \mu
    .
\end{align*}
The last equality uses the chain rule and symmetry of $H_P.$
Define the bilinear form
\begin{equation} \label{eq:bilinear}
    B_P (u, v)
    \coloneqq
    \int_\Omega
        \nabla u^\top H_P^{- 1} \nabla v
    d \mu
    ,
    \qquad
    u, v \in \cV
    .
\end{equation}
Writing $g_P \coloneqq \psi_P \circ T_P,$ the calculation yields $(R - P) \psi_P \approx B_P (g_P, \varphi_R - \varphi_P).$
Thus the desired approximation to $D_P (\varphi_R - \varphi_P)$ follows if we choose $g_P$ so that $B_P (g_P, v) = D_P (v)$ for every $v \in \cV,$ and then set $\psi_P \coloneqq g_P \circ F_P.$
This equation matches the two first-order terms exactly; the remaining task is to bound the two approximation errors.

The next proposition establishes existence and uniqueness of $g_P$ and the weak chain rule for $\psi_P.$
\begin{proposition}[Representation of linear forms] \label{prop:adjoint}
Under Assumption \ref{ass:base}, for each $D_P \in \cV^{\ast}$ there is a unique $g_P \in \cV$ such that
\begin{equation} \label{eq:adjoint}
    B_P (g_P, v)
    =
    D_P (v)
    \quad
    (v \in \cV)
    .
\end{equation}
The function $\psi_P \coloneqq g_P \circ F_P$ belongs to $H^1 (\Lambda),$ has $P \psi_P = 0,$ and satisfies
\begin{equation} \label{eq:chain}
    \nabla \psi_P (T_P)
    =
    H_P^{- 1} \nabla g_P
    \quad
    \mu \text{-almost everywhere}
    .
\end{equation}
\end{proposition}
The proof is in Appendix \ref{app:regularity}.
We call $g_P$ the adjoint of $D_P$ with respect to the bilinear form $B_P.$

We now impose separate bounds on the nonlinear functional remainder and the error in replacing the adjoint linear term by $(R - P) \psi_P.$
Both bounds hold for every comparison law on $\overline \Lambda,$ with constants independent of that law.
\begin{assumption}[Functional approximation] \label{ass:functional}
There exist $D_P \in \cV^{\ast},$ $r \in (1, 2],$ and $0 \leq C < \infty,$ independent of $R,$ such that the adjoint $g_P$ defined by (\ref{eq:adjoint}) satisfies (1)--(3):
\begin{enumerate}[(1)]
\item The adjoint $g_P$ has a continuous representative on $\overline \Omega.$
We use this representative to define $\psi_P \coloneqq g_P \circ F_P.$
\item The nonlinear functional remainder satisfies, for every probability law $R$ on $\overline \Lambda,$
\begin{equation} \label{eq:Jbound}
    | \cJ (\varphi_R)
    -
    \cJ (\varphi_P)
    -
    D_P (\varphi_R - \varphi_P) |
    \leq
    C \norm{T_R - T_P}_{L^2 (\mu)}^r
    .
\end{equation}
\item The adjoint remainder satisfies, for every probability law $R$ on $\overline \Lambda,$
\begin{equation} \label{eq:psibound}
    \left|
        \int_\Omega
            \left(
                \psi_P (T_R)
                -
                \psi_P (T_P)
                -
                \nabla \psi_P (T_P)^\top (T_R - T_P)
            \right)
        d \mu
    \right|
    \leq
    C \norm{T_R - T_P}_{L^2 (\mu)}^r
    .
\end{equation}
\end{enumerate}
\end{assumption}
Condition (1) gives $\psi_P = g_P \circ F_P$ a continuous, pointwise-defined representative, so $(R - P) \psi_P$ is well-defined even for discrete $R.$
Condition (2) requires a linear approximation to $\cJ$ whose error is of higher order than the $L^2 (\mu)$ map error, since $r > 1.$
By (\ref{eq:adjoint}) and (\ref{eq:chain}), condition (3) bounds the error in replacing this linear approximation by $(R - P) \psi_P.$

For $0 < \gamma \leq 1,$ condition (\ref{eq:Jbound}) holds with $r = 1 + \gamma$ and $D_P = D \cJ (\varphi_P)$ if $\cJ$ has a continuously Fr\'echet differentiable extension to an open convex neighborhood of $\cF$ in $\cV$ whose derivative is uniformly $\gamma$-H\"older continuous in the $\cV^\ast$ norm.
Likewise, $\psi_P \in C^{1, \gamma} (\overline \Lambda)$ implies (\ref{eq:psibound}) with $r = 1 + \gamma.$
Thus $C^{1, \alpha}$ regularity suffices for $r = 1 + \alpha,$ and Lipschitz derivatives give $r = 2.$

\section{Linearization of transport functionals} \label{sec:expansion}
We derive a linear expansion of the parameter in the target law, give sufficient conditions for its remainder bounds, and verify them for several classes of functionals.

\subsection{General expansion and map stability}
We first combine the functional and adjoint remainder bounds in Assumption \ref{ass:functional} to obtain a linear expansion in $R - P.$
Substantive sufficient conditions establishing these bounds are developed in Sections \ref{sec:riesz}–\ref{sec:examples}.
\begin{theorem}[Linear expansion] \label{thm:expansion}
Under Assumptions \ref{ass:base} and \ref{ass:functional}, for every probability law $R$ on $\overline \Lambda,$
\begin{equation} \label{eq:expansion}
    \theta (R)
    -
    \theta (P)
    =
    (R - P) \psi_P
    +
    \cR (R, P)
    ,
\end{equation}
where
$$
    | \cR (R, P) |
    \leq
    2C \norm{T_R - T_P}_{L^2 (\mu)}^r
    .
$$
\end{theorem}
The proof is in Appendix \ref{app:regularity}.

To control the remainder through the distance between $R$ and $P,$ we next bound the map error by $W_2 (R, P).$
\begin{lemma}[Map stability] \label{lem:stability}
Under Assumption \ref{ass:base}, there is a finite constant $C_P,$ independent of $R,$ such that, for every probability law $R$ on $\overline \Lambda,$
\begin{equation} \label{eq:stability}
    \norm{T_R - T_P}_{L^2 (\mu)}^2
    \leq
    C_P W_2^2 (R, P)
    .
\end{equation}
\end{lemma}
The proof is in Appendix \ref{app:regularity}.

Combining Theorem \ref{thm:expansion} and Lemma \ref{lem:stability} gives
$$
    | \cR (R, P) | \leq 2 C C_P^{r / 2} W_2^r (R, P).
$$
At $R = P_n,$ the condition $\sqrt n W_2^r (P_n, P) \to 0$ in probability therefore makes the remainder negligible at the root-$n$ scale.
Section \ref{sec:statistics} provides the required rates.

\subsection{Sufficient conditions} \label{sec:riesz}
We give sufficient conditions for the functional-approximation bounds in Assumption \ref{ass:functional}.
First, recall that for every $D_P \in \cV^{\ast},$ the Riesz representation theorem gives a unique $h_P \in \cV$ such that
\begin{equation} \label{eq:riesz}
    D_P (v)
    =
    \ip{h_P}{v}_\cV
    =
    \int_\Omega
        \nabla h_P^\top \nabla v
    d \mu
    \quad
    (v \in \cV)
    ,
    \qquad
    \norm{h_P}_\cV
    =
    \norm{D_P}_{\cV^{\ast}}
    .
\end{equation}
We call $h_P$ the Riesz representer with respect to the inner product of $\cV.$

Combining (\ref{eq:adjoint}) and (\ref{eq:riesz}), finding the adjoint is equivalent to finding $g_P \in H^1 (\Omega)$ with $\mu g_P = 0$ such that
$$
    \int_\Omega \nabla g_P^\top H_P^{- 1} \nabla v  d \mu = \int_\Omega \nabla h_P^\top \nabla v  d \mu, \qquad v \in H^1 (\Omega).
$$
Note that allowing all $H^1 (\Omega)$ test functions is equivalent to testing on $\cV,$ since subtracting $\mu v$ leaves the gradient unchanged.
This identity is the weak formulation of the Neumann boundary-value problem
\begin{equation} \label{eq:rieszpde}
    \begin{cases}
        \div (\rho H_P^{- 1} \nabla g_P)
        =
        \div (\rho \nabla h_P) & \text{in } \Omega
        \\
        \rho {(H_P^{- 1} \nabla g_P)^\top \nu}
        =
        \rho {(\nabla h_P)^\top \nu} & \text{on } \partial \Omega
        \\
        \mu g_P
        =
        0
    \end{cases}
    ,
\end{equation}
where $\nu$ is the outward unit normal.
The coefficient matrix $\rho H_P^{- 1}$ is uniformly elliptic by Lemma \ref{lem:mapregularity}.
The displayed equations are understood through the preceding weak identity.
For smooth $g_P$ and $h_P,$ compactly supported tests first recover the interior equation; using this equation, integration by parts with tests having arbitrary boundary values recovers the normal-flux condition.
For the standard weak formulation of elliptic equations and the Neumann problem, see \cite[Section 6.1.2 and Section 6.6, Problem 4]{evans2010partial}.

Smoothness of the Riesz representer ensures the adjoint regularity needed for the distributional expansion.
\begin{proposition}[Smooth Riesz representer] \label{prop:riesz}
Under Assumption \ref{ass:base}, suppose there exist $D_P \in \cV^{\ast},$ $r \in (1, 2],$ and a finite constant independent of $R$ for which (\ref{eq:Jbound}) holds for every probability law $R$ on $\overline \Lambda.$
If the Riesz representer of $D_P$ belongs to $C^{2, \alpha} (\overline \Omega),$ then $\cJ$ satisfies Assumption \ref{ass:functional} with the same $D_P$ and $r.$
\end{proposition}
The proof is in Appendix \ref{app:regularity}.
Elliptic regularity transfers smoothness from the prescribed functions in (\ref{eq:rieszpde}) to its solution.
When $h_P \in C^{2, \alpha} (\overline \Omega),$ the interior forcing belongs to $C^{0, \alpha}$ and the prescribed boundary flux to $C^{1, \alpha}.$
Together with the coefficient and boundary regularity in Assumption \ref{ass:base}, this yields $g_P \in C^{2, \alpha} (\overline \Omega)$ and hence $\psi_P = g_P \circ F_P \in C^{2, \alpha} (\overline \Lambda).$
Its bounded Hessian gives a quadratic Taylor remainder, which implies (\ref{eq:psibound}) for $r \leq 2$ because $\Lambda$ is bounded.

The smooth Riesz condition is equivalent to a representation of the first variation by smooth coefficients, which can be checked directly for the examples below.
\begin{proposition}[Smooth representations of the first variation] \label{prop:rieszfields}
Under Assumption \ref{ass:base}, let $D_P \in \cV^{\ast}$ and let $h_P$ be its Riesz representer.
Then $h_P \in C^{2, \alpha} (\overline \Omega)$ if and only if there exist $a_P \in C^{1, \alpha} (\overline \Omega; \bR^d)$ and $b_P \in C^{0, \alpha} (\overline \Omega)$ such that
\begin{equation} \label{eq:fields}
    D_P (v)
    =
    \int_\Omega
        (b_P v + a_P^\top \nabla v)
    d \mu
    \quad
    (v \in \cV)
    .
\end{equation}
\end{proposition}
The proof is in Appendix \ref{app:regularity}.

Thus the smooth Riesz condition means that the first-order sensitivity of the functional can be written as an average of changes in potential values and gradients against spatially smooth weights $b_P$ and $a_P.$
For the smooth examples below, we identify these weights and separately bound the nonlinear remainder; Propositions \ref{prop:rieszfields} and \ref{prop:riesz} then imply Assumption \ref{ass:functional}.

The smooth Riesz condition is sufficient but not necessary for Assumption \ref{ass:functional}.
The subgroup evaluations in Example \ref{ex:subgroup} need not have a smooth Riesz representer, and the maximum tail correlation in Example \ref{ex:tail} has a representer whose gradient jumps across the boundary of the tail region.
Both examples nevertheless satisfy Assumption \ref{ass:functional} with $r = 4 / 3,$ as shown by controlling the contribution near that boundary.

\subsection{Functional examples} \label{sec:examples}
We give several examples satisfying Assumption \ref{ass:functional}.
Throughout this subsection, suppose Assumption \ref{ass:base} holds.

We first consider integrals of the transport map against a vector weight.
\begin{example}[Linear evaluations of the transport map] \label{ex:linear}
For $w \in C^{1, \alpha} (\overline \Omega; \bR^d),$ define
\begin{equation} \label{eq:linearJ}
    \cJ_w (\varphi)
    \coloneqq
    \int_\Omega
        w (z)^\top \nabla \varphi (z)
    d \mu (z)
    ,
    \qquad
    \varphi \in \cF
    .
\end{equation}
Then $\cJ_w$ satisfies Assumption \ref{ass:functional} with $r = 2.$
\end{example}
Its first variation is
$$
    D_P (v)
    \coloneqq
    \int_\Omega
        w^\top \nabla v
    d \mu
    ,
    \qquad
    v \in \cV
    .
$$
The proof is in Appendix \ref{app:regularity}.
Proposition \ref{prop:rieszfields} applies with $a_P \coloneqq w$ and $b_P \coloneqq 0.$
The nonlinear functional remainder is zero, and Proposition \ref{prop:riesz} controls the adjoint remainder.

More general summaries can depend nonlinearly on both the potential value and the outcome assigned to each source type.
\begin{example}[Smooth integrals] \label{ex:local}
Let $\ell \in C^{2, \alpha} (\overline \Omega \times [- M, M] \times \overline \Lambda),$ where $M \coloneqq L_\Lambda \operatorname{diam} (\Omega).$
Define
\begin{equation} \label{eq:localJ}
    \cJ_\ell (\varphi)
    \coloneqq
    \int_\Omega
        \ell (z, \varphi (z), \nabla \varphi (z))
    d \mu (z)
    ,
    \qquad
    \varphi \in \cF
    .
\end{equation}
Then $\cJ_\ell$ satisfies Assumption \ref{ass:functional} with $r = 2.$
\end{example}
Its first variation is
$$
    D_P (v)
    \coloneqq
    \int_\Omega
        (b_P (z) v (z) + a_P (z)^\top \nabla v (z))
    d \mu (z)
    ,
    \qquad
    v \in \cV
    ,
$$
where
$$
    a_P (z)
    \coloneqq
    \nabla_y \ell (z, \varphi_P (z), T_P (z))
    \qquad
    \text{and}
    \qquad
    b_P (z)
    \coloneqq
    \partial_s \ell (z, \varphi_P (z), T_P (z))
    .
$$
The proof is in Appendix \ref{app:regularity}.

Squared Wasserstein distance is a smooth integral for which the influence function can be written explicitly using the convex conjugate $\varphi_P^{\ast}.$
\begin{example}[Wasserstein cost] \label{ex:cost}
The choice $\ell (z, s, y) \coloneqq \norm{z - y}^2$ in Example \ref{ex:local} gives
\begin{equation} \label{eq:costfunctional}
    \cJ_{\mathrm W} (\varphi)
    \coloneqq
    \int_\Omega
        \norm{z - \nabla \varphi (z)}^2
    d \mu (z)
    ,
    \qquad
    \cJ_{\mathrm W} (\varphi_R)
    =
    W_2^2 (\mu, R)
    .
\end{equation}
Its first variation has $b_P = 0$ and $a_P (z) = 2 (T_P (z) - z).$
Put
\begin{equation} \label{eq:costscore}
    \psi_{\mathrm W, P} (y)
    \coloneqq
    \norm{y}^2
    -
    2 \varphi_P^{\ast} (y)
    -
    \int_\Lambda
        \left(\norm{u}^2 - 2 \varphi_P^{\ast} (u)\right)
    d P (u)
    .
\end{equation}
The corresponding centered adjoint is $g_{\mathrm W, P} \coloneqq \psi_{\mathrm W, P} \circ T_P.$

Example \ref{ex:local} gives Assumption \ref{ass:functional} with $r = 2.$
Thus (\ref{eq:costscore}) gives the influence function for squared Wasserstein distance.
The proof is in Appendix \ref{app:regularity}.
\end{example}
Summaries of potential values, such as their dispersion, fit the same framework of integral functionals.
\begin{example}[Potential summaries] \label{ex:potential}
Let $U \in C^{2, \alpha} ([- M, M]).$
The choice $\ell (z, s, y) \coloneqq U (s)$ gives
$$
    \cJ_U (\varphi)
    \coloneqq
    \int_\Omega
        U (\varphi (z))
    d \mu (z)
    ,
    \qquad
    D_P (v)
    =
    \int_\Omega
        U^\prime (\varphi_P (z)) v (z)
    d \mu (z)
    .
$$
Here $a_P = 0$ and $b_P = U^\prime (\varphi_P).$
Hence $\cJ_U$ satisfies Assumption \ref{ass:functional} with $r = 2$ by Example \ref{ex:local}.
Section \ref{sec:limitcomparison} derives the same limit from the potential central limit theorem of \cite[Theorem 2.4(ii)]{gonzalez2026empirical} and the functional delta method.
\end{example}
Averages over a subgroup of source types introduce a discontinuous weight, so the effect of its boundary requires separate control.
\begin{example}[Exact subgroup evaluations] \label{ex:subgroup}
Let $A$ be a nonempty domain such that $\overline A \subset \Omega$ and $\partial A$ is $C^{2, 1}.$
For ${\ell} \in C^{2, \alpha} (\overline \Omega \times \overline \Lambda),$ define $\cJ_{A, {\ell}} : \cF \to \bR$ by
\begin{equation} \label{eq:subgroup}
    \cJ_{A, {\ell}} (\varphi)
    \coloneqq
    \frac 1{\mu (A)}
    \int_A
        {\ell} (z, \nabla \varphi (z))
    d \mu (z)
    .
\end{equation}
Then $\cJ_{A, {\ell}}$ satisfies Assumption \ref{ass:functional} with $r = 4 / 3.$
\end{example}
Its first variation is
$$
    D_P (v)
    =
    \frac 1{\mu (A)}
    \int_A
        \nabla_y {\ell} (z, T_P (z))^\top \nabla v (z)
    d \mu (z)
    ,
    \qquad
    v \in \cV
    .
$$
The proof is in Appendix \ref{app:subgroup}.

For this example, finding the adjoint amounts to solving a transmission problem, with one equation on each side of $\partial A.$
Put
$$
    A_P \coloneqq \rho H_P^{- 1},
    \qquad
    f_P (z) \coloneqq \frac{\rho (z)}{\mu (A)} \nabla_y \ell (z, T_P (z)).
$$
Write $g_{P, +}$ and $g_{P, -}$ for the restrictions of $g_P$ to $A$ and $\Omega \setminus \overline A,$ respectively, and let $\nu_A$ be the outward unit normal to $A.$
The adjoint is the weak solution, normalized by $\mu g_P = 0,$ of
\begin{equation} \label{eq:subgrouptransmission}
    \begin{cases}
        \div (A_P \nabla g_{P, +}) = \div f_P & \text{in } A, \\
        \div (A_P \nabla g_{P, -}) = 0 & \text{in } \Omega \setminus \overline A, \\
        g_{P, +} = g_{P, -} & \text{on } \partial A, \\
        {\bigl(A_P (\nabla g_{P, +} - \nabla g_{P, -})\bigr)^\top \nu_A} = {f_P^\top \nu_A} & \text{on } \partial A, \\
        {(A_P \nabla g_{P, -})^\top \nu} = 0 & \text{on } \partial \Omega.
    \end{cases}
\end{equation}
The interface conditions require continuity of $g_P$ and prescribe the jump in its normal flux.

Appendix \ref{app:subgroup} shows that $g_P$ is continuous on $\overline \Omega$ and is $C^{1, 1 / 2}$ separately on $\overline A$ and $\overline{\Omega \setminus \overline A},$ although its gradient may jump at $\partial A.$
The image $\Sigma_P \coloneqq T_P (\partial A)$ separates outcomes assigned to source points in $A$ from outcomes assigned to its complement.
Because $F_P$ maps these two regions back to their respective source regions, $\psi_P = g_P \circ F_P$ is continuous on $\overline \Lambda$ and is $C^{1, 1 / 2}$ on the closure of each region separately.
Its gradient may jump at $\Sigma_P.$
Controlling the Taylor remainder near this surface gives the exponent $r = 4 / 3.$

A useful subgroup summary averages the inner product of source ranks and matched outcomes over the outer region of the reference ball.
\begin{example}[Maximum tail correlation] \label{ex:tail}
Let $\Omega$ be the unit ball and $\mu$ the uniform distribution on it.
For a tail mass $\tau \in (0, 1),$ put $a_\tau \coloneqq (1 - \tau)^{1 / d},$ $A_\tau \coloneqq {B_{a_\tau}},$ and $E_\tau \coloneqq \Omega \setminus \overline A_\tau.$
Define
\begin{equation} \label{eq:tail}
    \kappa_\tau (P)
    \coloneqq
    \cJ_\tau (\varphi_P)
    ,
    \qquad
    \cJ_\tau (\varphi)
    \coloneqq
    \frac 1 \tau
    \int_{E_\tau}
        z^\top \nabla \varphi (z)
    d \mu (z)
    .
\end{equation}
This is the maximum tail correlation with a uniform reference distribution considered by \cite[Section 4.1.2]{sadhu2024stability}; it averages rank-weighted outcomes over the outer $\tau$ fraction of reference ranks.
The functional $\cJ_\tau$ in (\ref{eq:tail}) satisfies Assumption \ref{ass:functional} with $r = 4 / 3.$
\end{example}
The proof is in Appendix \ref{app:subgroup}.

The first variation is
$$
    D_\tau (v) \coloneqq \frac 1 \tau \int_{E_\tau} z^\top \nabla v  d \mu, \qquad v \in \cV.
$$
Its Riesz representer, with $t_+ \coloneqq \max (t, 0),$ is
$$
    h_\tau (z) \coloneqq \frac{1}{2 \tau} \left(\left(\norm{z}^2 - a_\tau^2\right)_+ - \int_\Omega \left(\norm{u}^2 - a_\tau^2\right)_+  d \mu (u)\right).
$$
This function is Lipschitz and has zero $\mu$-mean, so $h_\tau \in \cV.$
Its weak gradient is $\nabla h_\tau (z) = \tau^{-1} z \indic_{E_\tau} (z),$ which gives $\ip{h_\tau}{v}_\cV = D_\tau (v)$ for every $v \in \cV.$
This gradient jumps across $\partial A_\tau,$ and hence $h_\tau$ does not belong to $C^{2, \alpha} (\overline \Omega).$
Proposition \ref{prop:riesz} therefore does not apply.
The required remainder bound follows by writing the integral over $E_\tau$ as the integral over $\Omega$ minus that over $A_\tau,$ and applying Examples \ref{ex:linear} and \ref{ex:subgroup}.

Maximum tail correlation was originally introduced by \cite{beirlant2020center} using the spherical-uniform reference law, whose radius is uniform on $[0, 1].$
The present example uses the volume-uniform law on the ball and hence a different transport map.
The spherical-uniform density is proportional to $\norm{z}^{1 - d}$ and is unbounded at the origin for $d \geq 2,$ so that reference law is excluded by Assumption \ref{ass:base}.

\subsection{Connection with tested-map linearizations} \label{sec:testedmap}
The weighted-map central limit theorem of \cite[Proposition 4.5]{gonzalez2026influence} concerns the linear functional
\begin{equation} \label{eq:testedfunctional}
    \cJ_{P, g} (\varphi) \coloneqq \int_\Omega (H_P^{-1} \nabla g)^\top \nabla \varphi  d \mu, \qquad g \in C^{1, \eta} (\overline \Omega), \quad 0 < \eta \leq 1,
\end{equation}
whose weight is held unchanged as the target law varies.
Its first variation is $D_{P, g} (v) \coloneqq B_P (g - \mu g, v),$ so its adjoint is $g - \mu g$ and its influence function is $\psi_{P, g} \coloneqq (g - \mu g) \circ F_P.$
The functional remainder is zero, and $\psi_{P, g} \in C^{1, \eta} (\overline \Lambda)$ gives the adjoint remainder bound with $r = 1 + \eta.$
Theorem \ref{thm:expansion} therefore recovers their equation (73),
\begin{equation} \label{eq:testedexpansion}
    \cJ_{P, g} (\varphi_{P_n}) - \cJ_{P, g} (\varphi_P) = (P_n - P) \psi_{P, g} + O \left(\norm{T_{P_n} - T_P}_{L^2 (\mu)}^{1 + \eta}\right).
\end{equation}
For $d \in \{2, 3\}$ and $1 + \eta > d / 2,$ Theorem \ref{thm:empirical} below yields their Gaussian limit, with variance $\operatorname{Var}_\mu (g).$
Our results also cover nonlinear functionals and discontinuous weights, as in Examples \ref{ex:local} and \ref{ex:subgroup}.

For atomic contaminations $(1 - t) P + t \delta_y,$ \cite[Theorem 3.1 and equation (74)]{gonzalez2026influence} obtain a pointwise right derivative of the map away from $F_P (y)$ and identify its integral against $H_P^{-1} \nabla g$ with $\psi_{P, g} (y).$
That map derivative requires a separate analysis of a PDE with a Dirac source; our expansion concerns the scalar functional.

\section{Empirical CLTs, inference, and bootstrap} \label{sec:statistics}
We estimate $\theta (P)$ by evaluating the transport potential of the empirical target distribution.
The expansion in Section \ref{sec:expansion} expresses the estimation error as a centered sample average plus a transport remainder.
Controlling this remainder gives Gaussian limits, simultaneous inference for several functionals, and bootstrap confidence intervals under the dimension and remainder conditions below.

\subsection{Empirical estimator and remainder control}
We construct the empirical potential by substituting the empirical target law into the transport dual, and then apply $\cJ$ to this potential.
The observations satisfy the following sampling assumption.
\begin{assumption}[Sampling] \label{ass:sampling}
The observations $Y_1, \ldots, Y_n$ are independent and identically distributed with law $P.$
\end{assumption}
Define the empirical target law by $P_n \coloneqq n^{- 1} \sum_{i = 1}^n \delta_{Y_i}.$
Substituting $P_n$ into (\ref{eq:transportdual}) gives the linear program
\begin{equation} \label{eq:empiricaldual}
    (a_n, b_n) \in \argmin\limits_{\substack{
        a \in C (\overline \Omega),\ b \in \bR^n \\
        \forall z, i \ a (z) + b_i \geq z^\top Y_i
    }} \left(\int_\Omega a  d \mu + \frac1n \sum_{i = 1}^n b_i\right).
\end{equation}
Regard $b_n$ as a function on the empirical support, with $b_n (Y_i) = b_{n, i}.$
The observations are distinct almost surely because $P$ has a density.
By (\ref{eq:dualpotential}), we may choose $a_n = b_n^\ast,$ where
$$
    b_n^\ast (z) = \max_{1 \leq i \leq n} (z^\top Y_i - b_{n, i}), \qquad
    \varphi_{P_n} = b_n^\ast - \int_\Omega b_n^\ast  d \mu.
$$
The empirical plug-in estimator evaluates this normalized potential:
\begin{equation} \label{eq:estimator}
    \widehat \theta_n \coloneqq \cJ (\varphi_{P_n}).
\end{equation}
This empirical construction is also used by \cite[Section 3.2]{manole2024plugin} for transport maps and Wasserstein costs and by \cite{gonzalez2026empirical} for transport potentials.
\begin{remark}[Computation] \label{rem:computation}
Although (\ref{eq:empiricaldual}) is infinite-dimensional, eliminating $a$ reduces it to a convex optimization problem in the $n$ weights $b_i.$
This problem can be solved efficiently by established algorithms for semidiscrete optimal transport, such as the damped Newton method of \cite[Algorithm 1 and Theorem 1.5]{kitagawa2019convergence}.
\end{remark}
Theorem \ref{thm:expansion} bounds the distributional remainder by the empirical transport-map error, which we control next.
The following bound combines Lemma \ref{lem:stability} with \cite[Theorem 1.1]{ambrosio2022quadratic} in dimension two and \cite[Theorem 1.1(ii)]{huesmann2025asymptotics} in higher dimensions; Appendix \ref{app:empiricalrate} verifies their hypotheses.
\begin{lemma}[Empirical transport error] \label{lem:empiricalrate}
Under Assumptions \ref{ass:base} and \ref{ass:sampling}, for every $d \geq 2$ there is a finite constant $C$ independent of $n$ such that for all $n \geq 2,$
\begin{equation} \label{eq:maprate}
    \bE \left[
        \norm{T_{P_n} - T_P}_{L^2 (\mu)}^2
    \right]
    \leq
    Ca_{n, d}
    ,
    \qquad
    a_{n, d}
    \coloneqq
    \begin{cases}
        \log n / n
        ,
        & d
        =
        2
        ,
        \\
        n^{- 2 / d}
        ,
        & d
        \geq
        3
    \end{cases}
    .
\end{equation}
\end{lemma}
Combining this transport bound with Theorem \ref{thm:expansion} controls the empirical remainder and hence the bias of the plug-in estimator.
\begin{proposition}[Empirical expansion and remainder bound] \label{prop:empiricalbound}
Under Assumptions \ref{ass:base}, \ref{ass:functional}, and \ref{ass:sampling}, for every $d \geq 2$ there is a finite constant $C$ independent of $n$ such that for all $n \geq 2,$
\begin{equation} \label{eq:empiricalexp}
    \widehat \theta_n
    -
    \theta (P)
    =
    (P_n - P) \psi_P
    +
    \cR (P_n, P)
    ,
\end{equation}
where
$$
    \bE [| \cR (P_n, P) |]
    \leq
    C a_{n, d}^{r / 2}
    .
$$
\end{proposition}
The proof is in Appendix \ref{app:empiricalrate}.
The leading term of the right-hand side is the mean-zero sample average
$$
    (P_n - P) \psi_P
    =
    \frac 1n
    \sum_{i = 1}^n
        \psi_P (Y_i)
    ,
    \qquad
    P \psi_P = 0
    .
$$
Its fluctuation is $O_P (n^{- 1 / 2}),$ whereas the remainder is $O_P (a_{n, d}^{r / 2}) = o_P (1).$
Thus the estimator is consistent in every dimension.

At the root-$n$ scale, the available remainder bound is
\begin{equation} \label{eq:dimensionthreshold}
    \sqrt n  a_{n, d}^{r / 2}
    =
    \begin{cases}
        n^{(1 - r) / 2} (\log n)^{r / 2}
        ,
        & d
        =
        2
        ,
        \\
        n^{1 / 2 - r / d}
        ,
        & d
        \geq
        3
    \end{cases}
    .
\end{equation}
It tends to zero for $d = 2$ and $r > 1,$ or for $d = 3$ and $r > 3 / 2.$
In dimensions $d \geq 4,$ no $r \in (1, 2]$ makes this bound vanish.
The distributional expansion remains valid there, but this bound alone does not establish a root-$n$ approximation.

\subsection{Gaussian limits and joint inference} \label{sec:examplesinference}
When the scaled remainder $\sqrt{n} \cR (P_n, P)$ vanishes, the centered sample average determines the limiting distribution.
The following theorem gives the Gaussian limit when the transport remainder is negligible at the root-$n$ scale.
\begin{theorem}[Empirical central limit theorem] \label{thm:empirical}
Suppose Assumptions \ref{ass:base}, \ref{ass:functional}, and \ref{ass:sampling} hold, $d \in \{2, 3\},$ and $r > d / 2.$
Then
$$
    \sqrt n (\widehat \theta_n - \theta (P))
    =
    \frac{1}{\sqrt{n}}
    \sum_{i = 1}^n
        \psi_P (Y_i)
    +
    o_P (1)
    ,
$$
and consequently
\begin{equation} \label{eq:clt}
    \sqrt n (\widehat \theta_n - \theta (P))
    {\rightsquigarrow}
    N (0, \sigma_P^2)
    ,
\end{equation}
where
$$
    \sigma_P^2
    \coloneqq
    P \psi_P^2
    =
    \mu g_P^2
    .
$$
\end{theorem}
The proof is in Appendix \ref{app:empiricalrate}.
\begin{remark}[Degeneracy of the Gaussian limit] \label{rem:zerovariance}
The Gaussian limit in Theorem \ref{thm:empirical} is the point mass at zero precisely when the first variation vanishes:
$$
    \sigma_P^2 = 0 \quad \Longleftrightarrow \quad D_P (v) = 0 \text{ for every } v \in \cV.
$$
Indeed, $\sigma_P^2 = \mu g_P^2 = 0$ implies $g_P = 0$ and hence $D_P (v) = B_P (g_P, v) = 0$ for every $v \in \cV.$
Conversely, if $D_P$ vanishes on $\cV,$ uniqueness in Proposition \ref{prop:adjoint} gives $g_P = 0,$ so $\sigma_P^2 = 0.$

For example, the Wasserstein cost $\cJ_{\mathrm W} (\varphi) = \int_\Omega \norm{z - \nabla \varphi (z)}^2 d \mu (z)$ has first variation
$$
    D_P (v) = 2 \int_\Omega (T_P (z) - z)^\top \nabla v (z)  d \mu (z), \qquad v \in \cV.
$$
When $P = \mu,$ we have $T_P (z) = z,$ so $D_P = 0$ and the Gaussian limit is degenerate, i.e., $\sqrt{n} W_2^2 (P_n, P) \rightsquigarrow 0.$
See also \cite[discussion following Theorem 22]{manole2024plugin}.
\end{remark}
Under Assumptions \ref{ass:base} and \ref{ass:sampling} and the respective hypotheses of the examples, the linear vector-weight evaluations in Example \ref{ex:linear}, the smooth integrals in Example \ref{ex:local}, the Wasserstein cost in Example \ref{ex:cost}, and the potential summaries in Example \ref{ex:potential} have $r = 2.$
Theorem \ref{thm:empirical} therefore gives their population-centered Gaussian limits in dimensions two and three.

The exact subgroup evaluations in Example \ref{ex:subgroup} have $r = 4 / 3,$ so the same theorem applies in dimension two.
This also covers the maximum tail correlation in Example \ref{ex:tail}, which averages the inner product of source ranks and matched outcomes over the outer part of the reference ball and has the same remainder exponent.

The same expansions permit simultaneous inference for a finite collection of evaluations, with dependence determined by their influence functions.
\begin{corollary}[Several functionals] \label{cor:joint}
Suppose Assumptions \ref{ass:base} and \ref{ass:sampling} hold and $d \in \{2, 3\}.$
For an integer $m \geq 1,$ let $\cJ_1, \ldots, \cJ_m$ be functionals satisfying Assumption \ref{ass:functional} with exponents $r_j > d / 2$ and influence functions $\psi_{P, j}.$
For $\theta_j (P) \coloneqq \cJ_j (\varphi_P)$ and $\widehat \theta_{n, j} \coloneqq \cJ_j (\varphi_{P_n}),$
$$
    \sqrt n
    \left(
        \widehat \theta_{n, j}
        -
        \theta_j (P)
    \right)_{j = 1}^m
    {\rightsquigarrow}
    N_m (0, \Sigma)
    ,
    \qquad
    \Sigma_{jk}
    \coloneqq
    P (\psi_{P, j} \psi_{P, k})
    .
$$
\end{corollary}
The proof is in Appendix \ref{app:empiricalrate}.

\subsection{Bootstrap and confidence intervals}
We use the empirical bootstrap to construct confidence intervals from the Gaussian limits.
Conditionally on $Y_1, \ldots, Y_n,$ draw $Y_1^{\ast}, \ldots, Y_n^{\ast}$ independently from $P_n$ and set
$$
    P_n^{\ast}
    \coloneqq
    \frac 1n
    \sum_{i = 1}^n
        \delta_{Y_i^{\ast}}
    ,
    \qquad
    \widehat \theta_n^{\ast}
    \coloneqq
    \cJ (\varphi_{P_n^{\ast}})
    .
$$
The potential $\varphi_{P_n^{\ast}}$ uses the same normalization as $\varphi_{P_n},$ so that $\mu \varphi_{P_n^{\ast}} = 0.$

The centered bootstrap statistic is $S_n^{\ast} \coloneqq \sqrt n (\widehat \theta_n^{\ast} - \widehat \theta_n).$
Write $\bP^{\ast} (\cdot) \coloneqq \bP (\cdot \mid Y_1, \ldots, Y_n)$ and $\bE^{\ast}$ for conditional probability and expectation under the resampling law.
Let $\mathrm{BL}_1 (\bR)$ denote the functions $f : \bR \to \bR$ with $\norm{f}_\infty \leq 1$ and $\operatorname{Lip} (f) \leq 1.$
The bounded Lipschitz distance metrizes weak convergence, so convergence of the conditional bootstrap law in this distance, in probability, is the standard criterion for bootstrap consistency; see \cite[Section 3.6]{van1996weak}.

The following theorem establishes bootstrap consistency under the same conditions as the empirical central limit theorem and justifies the resulting confidence intervals.
\begin{theorem}[Empirical bootstrap] \label{thm:bootstrap}
Suppose Assumptions \ref{ass:base}, \ref{ass:functional}, and \ref{ass:sampling} hold, $d \in \{2, 3\},$ and $r > d / 2.$
For $Z \sim N (0, \sigma_P^2),$
$$
    \sup_{f \in \mathrm{BL}_1 (\bR)}
        \left|
            \bE^{\ast} [f (S_n^{\ast})]
            -
            \bE [f (Z)]
        \right|
    \xrightarrow{\bP}
    0
    .
$$
For $0 < u < 1,$ define the conditional quantile
$$
    q_{n, u}^{\ast}
    \coloneqq
    \inf
    \{t \in \bR \mid \bP^{\ast} (S_n^{\ast} \leq t) \geq u\}
    .
$$
If $\sigma_P^2 > 0,$ the interval
\begin{equation} \label{eq:CI}
    \left[
        \widehat \theta_n
        -
        \frac{q_{n, 1 - \gamma / 2}^{\ast}}{\sqrt n}
        ,
        \quad
        \widehat \theta_n
        -
        \frac{q_{n, \gamma / 2}^{\ast}}{\sqrt n}
    \right]
\end{equation}
has coverage converging to $1 - \gamma$ for every $0 < \gamma < 1.$
\end{theorem}
The proof is in Appendix \ref{app:empiricalrate}.
The same dimension and remainder conditions give both the Gaussian approximation and consistent bootstrap quantiles.

A practical advantage of the empirical bootstrap is that confidence intervals require neither solving the adjoint equation nor estimating the influence function or its variance.
Each resample only requires recomputing the empirical transport by the methods in Remark \ref{rem:computation} and evaluating $\cJ.$

\subsection{Comparison with existing limit theorems} \label{sec:limitcomparison}
We compare the scope of Theorem \ref{thm:empirical} with existing limit theorems for empirical potentials and Wasserstein costs.
The comparison concerns which scalar functionals are covered by potential convergence and when cost limits can be centered at the population value.

\noindent
\textbf{Potential limits.}
Potential-level limit theorems also yield scalar limits through the functional delta method.
Suppose additionally that $\varphi_P$ extends to a strongly convex function on $\bR^d$ that is $C^{3, \alpha}$ on a neighborhood of $\overline \Omega.$
Then \cite[Theorem 2.4(ii)]{gonzalez2026empirical} prove
\begin{equation} \label{eq:potentialCLT}
    \sqrt n (\varphi_{P_n} - \varphi_P) \rightsquigarrow Z_P \quad \text{in } L_0^q (\mu),
\end{equation}
where $Z_P$ is centered Gaussian, $1 < q < 2$ when $d = 2,$ and $1 < q < 6 / 5$ when $d = 3.$
For the potential summary in Example \ref{ex:potential}, $D_P (v) = \mu [U^\prime (\varphi_P) v]$ extends continuously to $L_0^q (\mu),$ and the functional delta method yields
$$
    \sqrt n \left(\cJ_U (\varphi_{P_n}) - \cJ_U (\varphi_P)\right)
    \rightsquigarrow D_P (Z_P).
$$
A calculation gives $D_P (Z_P) \sim N (0, P \psi_P^2),$ the same limiting law as in Theorem \ref{thm:empirical}.

On the other hand, an $L^q$ limit for potentials does not generally control functionals that depend on their gradients.
For the linear evaluation in Example \ref{ex:linear}, integration by parts gives
\par
$$
    D_P (v)
    =
    \int_\Omega
        \rho w^\top \nabla v
     d z
    =
    - \int_\Omega
        v  \div (\rho w)
     d z
    +
    \int_{\partial \Omega}
        v \rho {w^\top \nu}
     d S
    .
$$
Here $d S$ denotes surface measure.
The boundary integral is not controlled by the $L^q$ norm; subgroup evaluations have an analogous term on $\partial A.$
For Example \ref{ex:linear}, the boundary term is absent when $\rho {w^\top \nu} = 0$ on $\partial \Omega.$
Otherwise, the potential limit alone does not provide this delta-method argument.
Our expansion treats such evaluations through the adjoint and a remainder measured by the map error.

\noindent
\textbf{Wasserstein-cost limits.}
For $\theta (P) \coloneqq W_2^2 (\mu, P),$ the estimator is $W_2^2 (\mu, P_n).$
Under Assumptions \ref{ass:base} and \ref{ass:sampling}, Theorem \ref{thm:empirical} gives, in $d = 2, 3,$
\begin{equation} \label{eq:costclt}
    \sqrt n \left(W_2^2 (\mu, P_n) - W_2^2 (\mu, P)\right)
    {\rightsquigarrow}
    N (0, \sigma_{\mathrm W, P}^2)
    ,
\end{equation}
where $\sigma_{\mathrm W, P}^2 \coloneqq P \psi_{\mathrm W, P}^2$ and $\psi_{\mathrm W, P}$ is the centered cost-dual potential in (\ref{eq:costscore}).
This conclusion for the same empirical estimator is already implied by \cite[Theorems 3.2 and 5.5, Example 5.7]{hundrieser2024unifying}.
Note that their result does not require our density or boundary smoothness.
\cite[Theorem 22(iii)]{manole2024plugin} obtain the same population-centered limit for positive bounded densities on a common cube or flat torus in $d \leq 3,$ assuming a $C^2$ population potential.
Their higher-dimensional results use density-based estimators and additional smoothness.
Thus Example \ref{ex:cost} recovers an established limit, while our general expansion covers other evaluations of the matching.

In general dimension, \cite[Theorem 4.1]{del2019central} prove
$$
    \sqrt n
    \left(W_2^2 (\mu, P_n) - \bE [W_2^2 (\mu, P_n)]\right)
    {\rightsquigarrow}
    N (0, \sigma_{\mathrm W, P}^2)
    .
$$
Their positive-density, convex-support, and $(4 + \epsilon)$-moment assumptions hold under Assumption \ref{ass:base}.
The estimator and variance are the same, but the centering is its expectation rather than the population cost.
For every $d \geq 2$ and $n \geq 2,$ Assumptions \ref{ass:base} and \ref{ass:sampling} give
\begin{equation} \label{eq:costbias}
    0
    \leq
    \bE [W_2^2 (\mu, P_n)]
    -
    W_2^2 (\mu, P)
    \leq
    C a_{n, d}
    .
\end{equation}
The upper bound follows from Proposition \ref{prop:empiricalbound} with $r = 2;$ convexity of $R \mapsto W_2^2 (\mu, R)$ gives the lower bound by Jensen's inequality.
This permits population centering in $d = 2, 3$ because $\sqrt n a_{n, d} \to 0,$ but does not establish a negligible bias for $d \geq 4.$

\section{Efficiency} \label{sec:efficiency}
Under the conditions of Theorem \ref{thm:empirical}, the empirical plug-in estimator has a Gaussian limit with variance $P \psi_P^2.$
We show that this variance is the semiparametric efficiency bound, obtained by taking the supremum of the information bounds over one-dimensional perturbations of the target law.
The source law $\mu$ is known.

Fix $P$ satisfying Assumption \ref{ass:base}.
Following \cite[Example 3.11.1]{van1996weak}, consider the bounded, mean-zero scores
$$
    \cH_P \coloneqq \{\eta \in L^\infty (P) \mid P \eta = 0\},
$$
equipped with the $L^2 (P)$ inner product, and the submodels
\begin{equation} \label{eq:scorepath}
    d P_{t, \eta} = (1 + t \eta)  d P,
\end{equation}
for sufficiently small $t > 0.$
The functional $\theta$ is differentiable at $P$ relative to these submodels if there is a continuous linear form $\dot\theta_P$ on $\cH_P$ such that
$$
    \frac{\theta (P_{t, \eta}) - \theta (P)}{t}
    \longrightarrow \dot\theta_P (\eta),
    \qquad t \to 0,
    \qquad \eta \in \cH_P.
$$
Its semiparametric efficiency bound is
$$
    \sup_{\substack{\eta \in \cH_P \\ P \eta^2 > 0}}
    \frac{|\dot\theta_P (\eta)|^2}{P \eta^2}.
$$
The Gaussian law with this variance gives the convolution and local asymptotic minimax bounds stated below.

For each $\eta \in \cH_P,$ set $P_{n, \eta} \coloneqq P_{1 / \sqrt n, \eta}$ for sufficiently large $n.$
Write $\bE_{n, \eta}$ for expectation under $P_{n, \eta}^{\otimes n}.$
An estimator $\widetilde \theta_n$ is regular for these submodels if
$$
    \sqrt n \left(\widetilde \theta_n - \theta (P_{n, \eta})\right)
    \rightsquigarrow L_P
    \quad \text{under } P_{n, \eta}^{\otimes n},
    \qquad \eta \in \cH_P,
$$
for a limiting law $L_P$ independent of $\eta.$
A nonnegative loss $\ell$ on $\bR$ is bowl-shaped if it is lower semicontinuous and its sublevel sets are convex and symmetric about zero.
\begin{theorem}[Efficiency bounds and attainment] \label{thm:efficiency}
Suppose Assumptions \ref{ass:base} and \ref{ass:functional} hold.
Set $\sigma_P^2 \coloneqq P \psi_P^2$ and let $Z \sim N (0, \sigma_P^2).$
The following statements hold.
\begin{enumerate}[label=(\roman*)]
\item 
The functional $\theta$ is differentiable at $P$ relative to the submodels (\ref{eq:scorepath}), with $\dot\theta_P (\eta) = P (\eta \psi_P).$
Its semiparametric efficiency bound is $\sigma_P^2.$
\item 
Every regular estimator has limiting law $N (0, \sigma_P^2) \ast M_P^\prime$ for some probability law $M_P^\prime,$ where $\ast$ denotes convolution.
\item 
For every sequence of measurable estimators $\widetilde \theta_n : \overline \Lambda^n \to \bR$ and every bowl-shaped loss $\ell,$
\begin{equation} \label{eq:localminimax}
    \sup_{\substack{I \subset \cH_P \\ 0 < |I| < \infty}} \liminf_{n \to \infty} \sup_{\eta \in I}
    \bE_{n, \eta} \left[\ell \left(\sqrt n (\widetilde \theta_n - {\theta (P_{n, \eta})})\right)\right] \geq \bE [\ell (Z)].
\end{equation}
\item 
If, in addition, Assumption \ref{ass:sampling} holds, $d \in \{2, 3\},$ and $r > d / 2,$ then $\widehat \theta_n$ is regular with limiting law $N (0, \sigma_P^2).$
For every bounded continuous loss $\ell$ and every nonempty finite $I \subset \cH_P,$
$$
    \lim_{n \to \infty} \sup_{\eta \in I} \bE_{n, \eta} \left[\ell \left(\sqrt n (\widehat \theta_n - {\theta (P_{n, \eta})})\right)\right] = \bE [\ell (Z)].
$$
\par
\end{enumerate}
\end{theorem}
Under (iv), the empirical estimator attains the semiparametric variance bound $\sigma_P^2$ and the risk bound for bounded continuous bowl-shaped losses.
The proof is in Appendix \ref{app:efficiency}.

\section{Conclusion} \label{sec:conclusion}
We developed a direct approach to inference for scalar functionals of quadratic optimal transport with a known reference distribution.
A linear elliptic equation identifies the influence function, and transport-map error bounds control the remainder.
Under the stated conditions, this yields Gaussian limits, valid empirical bootstrap inference, and semiparametric efficiency of the empirical plug-in estimators.
The examples include smooth evaluations as well as subgroup and tail functionals with discontinuous weights.

These results permit scalar inference despite the absence of a nonzero weak limit for the normalized map error in $L^2.$
Extensions to higher dimensions and to an estimated reference distribution are left for future work.

\appendix

\section{Proofs} \label{sec:proofs}

\subsection{Transport regularity and linear expansions} \label{app:regularity}
\begin{proof}[Proof of Lemma \ref{lem:potentials}]
\setlength{\parindent}{24pt}
Brenier's theorem gives a convex potential with gradient $T_R.$
Positivity of $\rho$ makes this potential finite throughout $\Omega:$ every interior point lies in a simplex whose vertices belong to its full-measure finite domain.
Since $T_R \in \overline \Lambda$ almost everywhere, its gradient is bounded by $L_\Lambda.$
A finite convex function is locally Lipschitz, so its almost-everywhere gradient is also its weak gradient.
The bounded-gradient characterization of Lipschitz functions (see the proof of \cite[Section 5.8.2, Theorem 4]{evans2010partial}), applied on convex subdomains compactly contained in $\Omega,$ therefore shows that the potential is $L_\Lambda$-Lipschitz on $\Omega.$
It therefore extends continuously to $\overline \Omega.$
Subtracting its $\mu$-mean gives $\varphi_R \in \cV,$ and
$$
    | \varphi_R (z)|
    =
    \left|
        \int_\Omega
            (\varphi_R (z) - \varphi_R (u))
        d \mu (u)
    \right|
    \leq
    L_\Lambda \operatorname{diam} (\Omega)
    .
$$
Brenier uniqueness identifies the gradients of any two such potentials; Poincar\'e's inequality and the normalization then identify the potentials.
Equation (\ref{eq:normid}) follows from the definition of $\norm{\cdot}_\cV.$

We next prove continuity of $R \mapsto \varphi_R.$
Let $R_j {\rightsquigarrow} R.$
By \cite[Corollary 5.23]{villani2009optimal}, $T_{R_j} \to T_R$ in $\mu$-probability: the quadratic cost is continuous and bounded on $\overline \Omega \times \overline \Lambda,$ and Brenier's theorem gives a unique optimal plan from $\mu$ to $R.$
Since all these maps take values in the bounded set $\overline \Lambda,$ this convergence also holds in $L^2 (\mu).$
Equation (\ref{eq:normid}) then gives $\varphi_{R_j} \to \varphi_R$ in $\cV.$
\end{proof}
\begin{proof}[Proof of Proposition \ref{prop:adjoint}]
\setlength{\parindent}{24pt}
Lemma \ref{lem:mapregularity} implies
$$
    | B_P (u, v) |
    \leq
    m_P^{- 1} \norm{u}_\cV \norm{v}_\cV
    ,
    \qquad
    B_P (v, v)
    \geq
    M_P^{- 1} \norm{v}_\cV^2
    .
$$
Boundedness and coercivity on the Hilbert space $\cV$ give the existence and uniqueness of $g_P$ by Lax--Milgram (\cite[Section 6.2.1]{evans2010partial}).

We now prove that $\psi_P \in H^1 (\Lambda)$ and satisfies (\ref{eq:chain}).
Lemma \ref{lem:mapregularity} makes $F_P : \Lambda \to \Omega$ a $C^1$ diffeomorphism with $DF_P = (H_P \circ F_P)^{- 1}.$
The weak change-of-variables rule (\cite[Section 7.3, equation (7.19)]{gilbarg2001elliptic}) therefore gives
$$
    \nabla \psi_P = (DF_P)^\top (\nabla g_P \circ F_P) = (H_P^{- 1} \nabla g_P) \circ F_P
$$
almost everywhere, using symmetry of $H_P.$
Moreover, boundedness of $DF_P$ and $\det DT_P$ gives, by change of variables,
$$
    \int_\Lambda \left(| \psi_P|^2 + \norm{\nabla \psi_P}^2 \right) d y \leq C_P \int_\Omega \left(|g_P|^2 + \norm{\nabla g_P}^2 \right) d z < \infty.
$$
Thus $\psi_P \in H^1 (\Lambda),$ and composing the weak-gradient identity with $T_P$ proves (\ref{eq:chain}).
Finally $P \psi_P = \mu g_P = 0.$
\end{proof}
\begin{proof}[Proof of Theorem \ref{thm:expansion}]
\setlength{\parindent}{24pt}
Let $R$ be a probability law on $\overline \Lambda.$
Since $\varphi_R - \varphi_P \in \cV,$ the adjoint equation and Proposition \ref{prop:adjoint} give
\par
$$
    D_P (\varphi_R - \varphi_P)
    = B_P (g_P, \varphi_R - \varphi_P)
    = \int \nabla \psi_P (T_P)^\top (T_R - T_P) d \mu.
$$
To identify the remainder in (\ref{eq:expansion}), define
\begin{align*}
    E_J (R, P)& \coloneqq
    \cJ (\varphi_R) - \cJ (\varphi_P) - D_P (\varphi_R - \varphi_P), \\
    E_\psi (R, P)& \coloneqq
    \int \left(\psi_P (T_R) - \psi_P (T_P)
    - \nabla \psi_P (T_P)^\top (T_R - T_P) \right) d \mu.
\end{align*}
The push-forward identities imply
$$
    (R - P) \psi_P
    = D_P (\varphi_R - \varphi_P) + E_\psi (R, P).
$$
Therefore $\cR (R, P) = E_J (R, P) - E_\psi (R, P).$
The bounds (\ref{eq:Jbound}) and (\ref{eq:psibound}) give the result by the triangle inequality.
\end{proof}
\begin{proof}[Proof of Lemma \ref{lem:stability}]
\setlength{\parindent}{24pt}
We prove (\ref{eq:stability}) with $C_P = M_P / m_P,$ using the duality argument of \cite[Theorem 6]{manole2024plugin}.
For $y \in \overline \Lambda,$ the maximizer defining $\varphi_P^{\ast} (y)$ is $F_P (y).$
On the interior, differentiation gives $\nabla \varphi_P^{\ast} = F_P$ and
$D^2 \varphi_P^{\ast} = (H_P \circ F_P)^{- 1},$ with eigenvalues in $[M_P^{- 1}, m_P^{- 1}].$
These derivatives extend continuously from the interior to $\overline \Lambda.$

Define the dual gap by
$$
    G_P (z, y)
    \coloneqq
    \varphi_P (z)
    +
    \varphi_P^{\ast} (y)
    -
    z^\top y
    .
$$
Fenchel equality gives $G_P (z, T_P (z)) = 0$ and $\nabla_y G_P (z, T_P (z)) = 0.$
Taylor's formula therefore yields, for some $\xi$ on the segment joining $T_P (z)$ and $y,$
$$
    G_P (z, y) = \frac12 (y - T_P (z))^\top D^2 \varphi_P^\ast (\xi) (y - T_P (z)).
$$
The Hessian bounds imply
\par
$$
    \frac{\norm{y - T_P (z)}^2}{2 M_P}
    \leq
    G_P (z, y)
    \leq
    \frac{\norm{y - T_P (z)}^2}{2 m_P}
    .
$$
Minimizing its integral over $\Pi (\mu, R)$ is equivalent to minimizing quadratic cost.
If $\gamma$ optimally couples $P, R,$ then $(F_P, \id)_\# \gamma$ is a competitor and
$$
    \frac{\norm{T_R - T_P}_{L^2 (\mu)}^2}{2 M_P}
    \leq
    \int
        G_P (z, T_R (z))
    d \mu
    \leq
    \int
        G_P (F_P (y), y^\prime)
    d \gamma
    \leq
    \frac{W_2^2 (P, R)}{2 m_P}
    .
$$
This proves (\ref{eq:stability}) with $C_P \coloneqq M_P / m_P.$
\end{proof}
\begin{proof}[Proof of Proposition \ref{prop:riesz}]
\setlength{\parindent}{24pt}
Let $h_P \in \cV \cap C^{2, \alpha} (\overline \Omega)$ be the Riesz representer of the first variation in (\ref{eq:Jbound}).
We construct a smooth solution of the boundary-value problem (\ref{eq:rieszpde}) and then verify the adjoint identity (\ref{eq:adjoint}) by integration by parts.
Define
$$
    A_P
    \coloneqq
    \rho H_P^{- 1}
    ,
    \qquad
    f_P
    \coloneqq
    \div (\rho \nabla h_P)
    ,
    \qquad
    k_P
    \coloneqq
    \rho {(\nabla h_P)^\top \nu}
    .
$$
Lemma \ref{lem:mapregularity} gives $A_P \in C^{1, \alpha} (\overline \Omega)$ and $\xi^\top A_P \xi \geq (\inf \rho) M_P^{- 1} \norm{\xi}^2.$
Since the boundary normal belongs to $C^{1, 1} (\partial \Omega; \bR^d),$ we have $f_P \in C^{0, \alpha} (\overline \Omega)$ and $k_P \in C^{1, \alpha} (\partial \Omega).$
The divergence theorem gives $\int_\Omega f_P d z = \int_{\partial \Omega} k_P d S.$

The prescribed boundary function $k_P$ must be the restriction of a function in $C^{1, \alpha} (\overline \Omega).$
Since $\Omega$ has $C^{1, \alpha}$ boundary, \cite[Lemma 6.38]{gilbarg2001elliptic} gives an extension of $k_P$ in $C^{1, \alpha} (\overline \Omega).$
Now, applying \cite[Theorem 2.1]{kono2025well} to (\ref{eq:rieszpde}) gives a unique solution $\tilde g_P \in C^{2, \alpha} (\overline \Omega)$ of (\ref{eq:rieszpde}) such that $\mu \tilde g_P = 0.$

Integration by parts gives, for smooth $v,$
\begin{align*}
    \int_\Omega
        \nabla \tilde g_P^\top H_P^{- 1} \nabla v
    d \mu
    &=
    -
    \int_\Omega
        f_P v
    d z
    +
    \int_{\partial \Omega}
        k_P v
    d S
    \\
    &=
    \int_\Omega
        \nabla h_P^\top \nabla v
    d \mu
    =
    D_P (v)
    ,
    \quad
    v \in \cV
    .
\end{align*}
The identities extend to $H^1$ test functions by smooth approximation.
Uniqueness in Proposition \ref{prop:adjoint} identifies this solution with its adjoint, i.e., $\tilde g_P = g_P.$
This shows (1) of Assumption \ref{ass:functional}.

Composition with $F_P \in C^{2, \alpha} (\overline \Lambda; \bR^d)$ yields $\psi_P = g_P \circ F_P \in C^{2, \alpha} (\overline \Lambda),$ since
$$
    D^2 \psi_P
    =
    (DF_P)^\top (D^2 g_P \circ F_P) DF_P
    +
    \sum_{j = 1}^d
        (\partial_j g_P \circ F_P) D^2 F_{P, j}
    ,
$$
where every term is bounded.
Taylor's formula gives
$$
    \left|
        \int
            \left(
                \psi_P (T_R)
                -
                \psi_P (T_P)
                -
                \nabla \psi_P (T_P)^\top (T_R - T_P)
            \right)
        d \mu
    \right|
    \leq
    \tfrac 12 \norm{D^2 \psi_P}_\infty \norm{T_R - T_P}_{L^2 (\mu)}^2
    .
$$
Since both $T_R, T_P$ take values in $\overline \Lambda,$ $\norm{T_R - T_P}_{L^2 (\mu)} \leq \operatorname{diam} (\Lambda).$
Consequently
$$
    \norm{T_R - T_P}_{L^2 (\mu)}^2
    \leq
    \operatorname{diam} (\Lambda)^{2 - r} \norm{T_R - T_P}_{L^2 (\mu)}^r
    .
$$
This proves (\ref{eq:psibound}) for every $R.$
Together with the assumed (\ref{eq:Jbound}), it proves Assumption \ref{ass:functional}.
\end{proof}
\begin{proof}[Proof of Proposition \ref{prop:rieszfields}]
\setlength{\parindent}{24pt}
Suppose (\ref{eq:fields}) holds with the stated smooth fields.
Set $b_P^\circ \coloneqq b_P - \mu b_P,$ which preserves $D_P$ on $\cV.$
To prove $h_P \in C^{2, \alpha} (\overline \Omega),$ we construct a smooth function with the same Riesz identity.
Consider the boundary-value problem
\begin{equation} \label{eq:rieszfields}
    \begin{cases}
        \div (\rho \nabla {\widetilde h_P})
        =
        \div (\rho a_P)
        -
        \rho b_P^\circ & \text{in } \Omega
        \\
        \rho {(\nabla \widetilde h_P)^\top \nu}
        =
        \rho {a_P^\top \nu} & \text{on } \partial \Omega
        \\
        \mu {\widetilde h_P}
        =
        0
    \end{cases}
    .
\end{equation}
Its coefficient $\rho I_d$ is uniformly elliptic and belongs to $C^{1, \alpha} (\overline \Omega).$
The right-hand side $\div (\rho a_P) - \rho b_P^\circ $ of the interior equation belongs to $C^{0, \alpha} (\overline \Omega),$ the prescribed boundary function $\rho {a_P^\top \nu}$ belongs to $C^{1, \alpha} (\partial \Omega),$ and
$$
    \int_\Omega
        (\div (\rho a_P) - \rho b_P^\circ )
    d z
    =
    \int_{\partial \Omega}
        \rho {a_P^\top \nu}
    d S
    .
$$
The boundary datum extends to $C^{1, \alpha} (\overline \Omega)$ by \cite[Lemma 6.38]{gilbarg2001elliptic}.
Thus \cite[Theorem 2.1]{kono2025well} gives a solution $\widetilde h_P \in C^{2, \alpha} (\overline \Omega)$ with $\mu \widetilde h_P = 0.$

For smooth $v \in \cV,$ integration by parts gives
\par
$$
    \int
        \nabla \mathord{\widetilde h_P}^\top \nabla v
    d \mu
    =
    \int
        (b_P^\circ v + a_P^\top \nabla v)
    d \mu
    =
    D_P (v)
    ,
    \quad
    v \in \cV
    .
$$
Smooth approximation extends this identity to every $v \in \cV.$
Uniqueness in the Riesz representation theorem gives $\widetilde h_P = h_P,$ so $h_P \in C^{2, \alpha} (\overline \Omega).$

Conversely, a representer in $C^{2, \alpha} (\overline \Omega)$ gives (\ref{eq:fields}) with $a_P = \nabla h_P, b_P = 0.$
This proves the equivalence.
\end{proof}
\begin{proof}[Proof of Example \ref{ex:linear}]
\setlength{\parindent}{24pt}
Cauchy--Schwarz gives $| D_P (v) | \leq \norm{w}_{L^2 (\mu)} \norm{v}_\cV.$
Linearity gives
$$
    \cJ_w (\varphi_R)
    -
    \cJ_w (\varphi_P)
    =
    D_P (\varphi_R - \varphi_P)
    ,
$$
so (\ref{eq:Jbound}) holds with zero remainder.
Proposition \ref{prop:rieszfields}, with $a_P = w$ and $b_P = 0,$ gives a Riesz representer in $C^{2, \alpha} (\overline \Omega)$.
Proposition \ref{prop:riesz} therefore gives Assumption \ref{ass:functional} with $r = 2.$
\end{proof}
\begin{proof}[Proof of Example \ref{ex:local}]
\setlength{\parindent}{24pt}
We first verify $D_P \in \cV^{\ast}$ for the stated coefficients $a_P$ and $b_P.$
Poincar\'e's inequality gives
$$
    | D_P (v) |
    \leq
    C
    \left(
        \norm{b_P}_{L^2 (\mu)}
        +
        \norm{a_P}_{L^2 (\mu)}
    \right) \norm{v}_\cV
    .
$$
Since the two coefficients $a_P, b_P$ are bounded in $L^2 (\mu),$ $D_P \in \cV^{\ast}.$

We next prove (\ref{eq:Jbound}) with $r = 2.$
Let
$$
    L_\ell
    \coloneqq
    \sup_{
        \substack{z \in \overline \Omega, \  | s | \leq M \\ y \in \overline \Lambda}
    } \norm{D^2_{(s, y)} \ell (z, s, y)}
    ,
$$
where the Hessian is taken with respect to the combined argument $(s, y).$
This constant is finite.
For every target law $R,$ Lemma \ref{lem:potentials}
puts $\varphi_P (z)$ and $\varphi_R (z)$ in $[ - M, M],$ while
$T_P (z), T_R (z) \in \overline \Lambda$ almost everywhere.
Convexity of the interval
and of $\overline \Lambda$ keeps the segment between these argument pairs in the
same set.
Taylor's formula in $(s, y),$ integrated with respect to $\mu,$ gives
\begin{align*}
    \left| 
        \cJ_\ell (\varphi_R)
        -
        \cJ_\ell (\varphi_P)
        -
        D_P (\varphi_R - \varphi_P)
    \right|
    & \leq
    \frac{L_\ell}{2} \left(
        \norm{\varphi_R - \varphi_P}_{L^2 (\mu)}^2
        +
        \norm{T_R - T_P}_{L^2 (\mu)}^2
    \right)
    \\
    & \leq
    \frac{L_\ell}{2} (1 + C) \norm{T_R - T_P}_{L^2 (\mu)}^2
    ,
\end{align*}
where the second inequality follows by Poincar\'e's inequality.
This proves (\ref{eq:Jbound}) with a constant independent of $R.$

To verify the smooth Riesz condition, the chain rule gives
$$
    D a_P (z) = D_z \nabla_y \ell + \partial_s \nabla_y \ell T_P (z)^\top + D_y^2 \ell H_P (z),
$$
where all derivatives of $\ell$ are evaluated at $(z, \varphi_P (z), T_P (z)).$
Each term is $\alpha$-H\"older on $\overline \Omega$ by Assumption \ref{ass:base} and the regularity of $\ell.$
Thus $a_P \in C^{1, \alpha} (\overline \Omega; \bR^d).$
The same composition argument gives $b_P \in C^{0, \alpha} (\overline \Omega).$
Proposition \ref{prop:rieszfields} gives $h_P \in C^{2, \alpha} (\overline \Omega).$
Together with the functional bound, Proposition \ref{prop:riesz} yields Assumption \ref{ass:functional}.
\end{proof}
\begin{proof}[Proof of Example \ref{ex:cost}]
\setlength{\parindent}{24pt}
We first verify that $g_{\mathrm W, P}$ is the adjoint and $\psi_{\mathrm W, P}$ is the influence function.
Example \ref{ex:local}, with $\ell (z, s, y) \coloneqq \norm{z - y}^2,$ gives the first variation with $b_P = 0$ and $a_P (z) = 2 (T_P (z) - z).$
Since $\nabla \varphi_P^{\ast} = F_P,$ the influence function in (\ref{eq:costscore}) satisfies $\nabla \psi_{\mathrm W, P} (T_P (z)) = 2 (T_P (z) - z).$
The chain rule therefore gives $B_P (g_{\mathrm W, P}, v) = D_P (v)$ for every $v \in \cV,$ and $\mu g_{\mathrm W, P} = P \psi_{\mathrm W, P} = 0.$
This identifies the adjoint.

To obtain the cost expansion and its remainder bound, use the dual gap $G_P$ from the proof of Lemma \ref{lem:stability}.
For $z \in \overline \Omega$ and $y \in \overline \Lambda,$
$$
    \norm{z - y}^2
    =
    \left(\norm{z}^2 - 2 \varphi_P (z)\right)
    +
    \left(\norm{y}^2 - 2 \varphi_P^{\ast} (y)\right)
    +
    2 G_P (z, y)
    .
$$
Integrating at $T_R$ and $T_P$ gives
\begin{align}
    W_2^2 (\mu, R)
    -
    W_2^2 (\mu, P)
    & =
    (R - P) \psi_{\mathrm W, P}
    +
    \cR_{\mathrm W} (R, P)
    ,
    \label{eq:costexpansion}
    \\
    \cR_{\mathrm W} (R, P)
    & =
    2
    \int_\Omega
        G_P (z, T_R (z))
    d \mu (z)
    ,
    \notag
    \\
    \frac 1{M_P} \norm{T_R - T_P}_{L^2 (\mu)}^2
    & \leq
    \cR_{\mathrm W} (R, P)
    \leq
    \frac 1{m_P} \norm{T_R - T_P}_{L^2 (\mu)}^2
    .
    \label{eq:costgapbound}
\end{align}
The inequalities follow from the Taylor bounds in the proof of Lemma \ref{lem:stability}.
\end{proof}

\subsection{Subgroup evaluations} \label{app:subgroup}
Suppose Assumption \ref{ass:base} holds.
Consider a nonempty domain $A$ such that $\overline A \subset \Omega$ and $\partial A$ is $C^{2, 1}.$
For $c_P \in C^{1, \alpha} (\overline \Omega; \bR^d),$ Cauchy--Schwarz gives
$$
    \left|\int_A c_P^\top \nabla v d \mu\right|
    \leq \norm{c_P \indic_A}_{L^2 (\mu)} \norm{v}_{\cV},
    \qquad v \in \cV.
$$
Proposition \ref{prop:adjoint} therefore gives a unique $g_P \in \cV$ satisfying
\begin{equation}\label{eq:subgroupadjoint}
    B_P (g_P, v) = \int_A c_P^\top \nabla v d \mu,
    \qquad v \in \cV.
\end{equation}
Put $A_P \coloneqq \rho H_P^{-1}$ and $f_P \coloneqq \rho c_P.$
By Lemma \ref{lem:mapregularity}, $A_P$ is uniformly elliptic and belongs to $C^{1, \alpha} (\overline \Omega; \bR^{d \times d}).$
For any $v \in H^1 (\Omega),$ the function $v - \mu v$ lies in $\cV$ and has the same gradient as $v.$
Using it in (\ref{eq:subgroupadjoint}) yields
\begin{equation}\label{eq:subgroupweak}
    \int_\Omega (A_P \nabla g_P)^\top \nabla v d z
    = \int_\Omega (f_P \indic_A)^\top \nabla v d z,
    \qquad v \in H^1 (\Omega).
\end{equation}
For $r > 0,$ let $B_r (a)$ denote the open ball of radius $r$ centered at $a,$ and set $B_r^+ \coloneqq B_r \cap \{x \mid x_d > 0\}$ and $B_r^- \coloneqq B_r \cap \{x \mid x_d < 0\}.$

We first control the adjoint near $\partial A,$ across which the field $f_P \indic_A$ in (\ref{eq:subgroupweak}) changes from $f_P$ to zero.
\begin{claim}[Interface regularity]\label{clm:subgroup-interface}
For every $a \in \partial A,$ there is an open neighborhood $U_a\subset\Omega$ such that the restrictions of $g_P$ to $A \cap U_a$ and $(\Omega \setminus \overline A) \cap U_a$ have $C^{1, 1 / 2}$ representatives on their respective closures.
\end{claim}
\begin{proof}[Proof of Claim \ref{clm:subgroup-interface}]
\setlength{\parindent}{24pt}
Let $a \in \partial A.$
Choose $r > 0$ such that $\overline{B_{2 r} (a)} \subset \Omega$ and $\partial A$ is a single $C^{2, 1}$ graph in this ball.
For tests supported there, (\ref{eq:subgroupweak}) implies that $g = g_P$ is a weak solution of 
$$
    \div (A_P \nabla g) = \div (f_P \indic_A)
    \quad \text{in } B_{2 r} (a)
    .
$$
The coefficient $A_P$ is uniformly elliptic, and the fields $A_P,$ $f_P,$ and zero are Lipschitz, hence $C^{0, 1 / 2},$ on the closures of the respective sides.
The interface is $C^{1, 1 / 2}.$
After translation and rescaling, \cite[Corollary 3.9]{zhuge2021regularity} applies directly with $k = 1,$ H\"older exponent $1 / 2,$ one scalar unknown, coefficient field $A_P,$ and right-hand vector field $f_P \indic_A.$
It gives the claimed regularity with $U_a = B_r (a).$
\end{proof}
We next establish regularity at the outer boundary $\partial \Omega.$
Since $\operatorname{dist} (\overline A, \partial \Omega)>0,$ there is an open neighborhood $U$ of $\partial \Omega$ disjoint from $\overline A.$
Consequently, the vector field $f_P \indic_A$ on the right-hand side of (\ref{eq:subgroupweak}) is identically zero on $U \cap \Omega.$
\begin{claim}[Boundary regularity]\label{clm:subgroup-boundary}
For every $a \in \partial \Omega,$ there is an open neighborhood $U_a$ disjoint from $\overline A$ such that
$$
    g_P|_{\Omega \cap U_a} \in C^{1, 1 / 2} (\overline{\Omega \cap U_a}).
$$
\end{claim}
\begin{proof}[Proof of Claim \ref{clm:subgroup-boundary}]
\setlength{\parindent}{24pt}
Fix $a \in \partial \Omega.$
We prove regularity near $a$ by straightening the boundary and reflecting the transformed equation across the resulting flat boundary.
The reflected equation is an interior interface problem to which \cite[Corollary 3.9]{zhuge2021regularity} applies.

\noindent
\textbf{Step 1.}
We first construct a map $\Phi_a$ that sends a flat half-ball into $\Omega$ and its flat boundary onto $\partial \Omega$ near $a.$
Thus $\Phi_a^{-1}$ straightens the boundary, so that reflection can be performed across $x_d = 0.$
Choose an orthogonal matrix $Q_a$ whose first $d - 1$ columns span the tangent plane to $\partial \Omega$ at $a$ and whose last column is the inward normal $-\nu (a).$
In the rotated variables $\xi = Q_a^\top (z - a),$ the boundary is locally the graph $\xi_d = h_a (\xi_{-d}),$ where $h_a \in C^{2, 1},$ $h_a (0) = 0,$ and $\nabla h_a (0) = 0;$ the domain lies on the side $\xi_d > h_a (\xi_{-d}).$
Choose $r > 0$ sufficiently small that the map
$$
    \Phi_a (x) \coloneqq a + Q_a \begin{pmatrix}x_{-d} \\ x_d + h_a (x_{-d})\end{pmatrix}, \qquad x \in B_{2 r},
$$
is defined on a neighborhood of $\overline{B_{2 r}},$ its image lies in this graph neighborhood, and its image is disjoint from $\overline A.$
The map $x_{-d} \mapsto \Phi_a (x_{-d}, 0)$ parametrizes the boundary surface, while $x_d > 0$ corresponds to points inside $\Omega.$
Its inverse sends $z$ to $(\xi_{-d}, \xi_d - h_a (\xi_{-d})),$ with $\xi = Q_a^\top (z - a).$
Hence $\Phi_a$ is a $C^{2, 1}$ diffeomorphism onto its image and
$$
    \Phi_a (B_{2 r}^+) = \Omega \cap \Phi_a (B_{2 r}), \qquad
    \Phi_a (B_{2 r} \cap \{x \mid x_d = 0\}) = \partial \Omega \cap \Phi_a (B_{2 r}).
$$

\noindent
\textbf{Step 2.}
We next express (\ref{eq:subgroupweak}) in these variables, retaining test functions that need not vanish on the flat boundary.
Set $u \coloneqq g_P \circ \Phi_a.$
The weak chain rule gives $u \in H^1 (B_{2 r}^+)$; see \cite[Section 7.3, formula (7.19)]{gilbarg2001elliptic}.
Writing
$$
    J \coloneqq D \Phi_a = Q_a \begin{pmatrix}I_{d - 1} & 0 \\ (\nabla h_a)^\top & 1\end{pmatrix}, \qquad |\det J| = 1,
$$
and using $\nabla g_P \circ \Phi_a = (J^{-1})^\top \nabla u,$ the transformed coefficient is
$$
    \widetilde A \coloneqq J^{-1} (A_P \circ \Phi_a) (J^{-1})^\top.
$$
Both $J$ and $J^{-1}$ are bounded and Lipschitz, so $\widetilde A$ is Lipschitz.
If $\lambda > 0$ is an ellipticity constant for $A_P,$ then
$$
    v^\top \widetilde A (x) v \geq \lambda \norm{(J (x)^{-1})^\top v}^2
    \geq \lambda \left(\sup_{B_{2 r}} \norm{J}\right)^{-2} \norm{v}^2.
$$
Thus $\widetilde A$ is uniformly elliptic.
For $\eta \in C_c^\infty (B_{2 r}),$ use $\eta \circ \Phi_a^{-1}$ on $\Omega \cap \Phi_a (B_{2 r})$ and extend it by zero to the rest of $\Omega.$
This is an $H^1 (\Omega)$ test function because $\eta$ vanishes near the spherical boundary of the half-ball; it need not vanish on $\partial \Omega.$
Its support is disjoint from $\overline A,$ so (\ref{eq:subgroupweak}) and a change of variables give
\begin{equation} \label{eq:subgroup-boundary-local}
    \int_{B_{2 r}^+} (\widetilde A \nabla u)^\top \nabla \eta  d x = 0,
    \qquad \eta \in C_c^\infty (B_{2 r}).
\end{equation}

\noindent
\textbf{Step 3.}
We now extend $u$ and its equation to the full ball, making $x_d = 0$ an interior interface.
Let $S \coloneqq \operatorname{diag} (1, \ldots, 1, -1)$ and define
$$
    u^{\mathrm e} (x) \coloneqq u (x_{-d}, |x_d|), \qquad
    \widetilde A^{\mathrm e} (x) \coloneqq
    \begin{cases}
        \widetilde A (x), & x_d > 0, \\
        S \widetilde A (S x) S, & x_d < 0.
    \end{cases}
$$
Approximate $u$ in $H^1 (B_{2 r}^+)$ by functions smooth up to the boundary and reflect them evenly.
The squared $H^1$ norm of each reflected difference is twice its squared half-ball norm.
Their limit is therefore $u^{\mathrm e} \in H^1 (B_{2 r}),$ with gradient $\nabla u$ above the interface and $S \nabla u (S x)$ below it.
For $\zeta \in C_c^\infty (B_{2 r}),$ taking $\eta = \zeta + \zeta \circ S$ in (\ref{eq:subgroup-boundary-local}) yields
$$
    \int_{B_{2 r}} (\widetilde A^{\mathrm e} \nabla u^{\mathrm e})^\top \nabla \zeta  d x
    = \int_{B_{2 r}^+} (\widetilde A \nabla u)^\top \nabla (\zeta + \zeta \circ S)  d x = 0.
$$
The reflected coefficient is uniformly elliptic, since $S$ is orthogonal, and is Lipschitz on each closed half-ball.
After rescaling, \cite[Corollary 3.9]{zhuge2021regularity} applies with $k = 1,$ H\"older exponent $1 / 2,$ a flat interface, and zero right-hand vector field.
It gives $u \in C^{1, 1 / 2} (\overline{B_r^+}).$
Since $g_P = u \circ \Phi_a^{-1}$ and $\Phi_a^{-1}$ is $C^{2, 1},$ the claim follows with $U_a = \Phi_a (B_r).$
\end{proof}
The two boundary estimates, together with interior regularity, yield the following global bounds for the subgroup adjoint.
\begin{lemma}[Regularity of the subgroup adjoint]\label{lem:subgroupregularity}
Under Assumption \ref{ass:base}, let $A$ be a nonempty domain such that $\overline A \subset \Omega$ and $\partial A$ is $C^{2, 1}.$
For every $c_P \in C^{1, \alpha} (\overline \Omega; \bR^d),$ the solution $g_P$ of (\ref{eq:subgroupadjoint}) has a globally Lipschitz representative on $\overline \Omega$ and satisfies
$$
    g_P|_A \in C^{1, 1 / 2} (\overline A),
    \qquad
    g_P|_{\Omega \setminus \overline A}
       \in C^{1, 1 / 2} (\overline{\Omega \setminus \overline A}).
$$
\end{lemma}
\begin{proof}[Proof of Lemma \ref{lem:subgroupregularity}]
\setlength{\parindent}{24pt}
For any ball $B$ with $\overline B \subset A,$ tests supported in $B$ reduce (\ref{eq:subgroupweak}) to
$$
    \int_B {(A_P \nabla g_P)^\top \nabla \zeta} d z
    = \int_B {f_P^\top \nabla \zeta} d z,
    \qquad \zeta \in C_c^\infty (B).
$$
If instead $\overline B \subset \Omega \setminus \overline A,$ the same identity holds with zero on the right-hand side.
Since $A_P$ and $f_P$ are Lipschitz on these balls, the interior Schauder estimate for functions satisfying this identity gives $g_P \in C^{1, 1 / 2}$ on smaller closed balls; see \cite[Section 8.11]{gilbarg2001elliptic}.

A finite cover of $\overline \Omega$ by these balls and the neighborhoods in Claims \ref{clm:subgroup-interface} and \ref{clm:subgroup-boundary} shows that $g_P$ and its weak gradient are essentially bounded.
The bounded weak-gradient characterization of Lipschitz functions (\cite[Section 5.8.2, Theorem 4]{evans2010partial}) therefore gives a Lipschitz representative on the bounded convex domain $\Omega.$
It extends uniquely to a Lipschitz function on $\overline \Omega.$

Consider either closed set $K = \overline A$ or $K = \overline{\Omega \setminus \overline A}.$
On overlapping neighborhoods, the continuous gradient representatives agree in the interior because both equal the weak gradient almost everywhere; continuity then gives agreement on the boundary as well.
They therefore define one continuous gradient on $K.$

Shrink the finitely many neighborhoods if necessary while retaining a cover of $K.$
Compactness gives $\delta > 0$ such that any two points of $K$ at distance less than $\delta$ belong to a common original neighborhood.
Let $C_0$ be the largest of the finitely many H\"older constants and let $M_0$ bound the gradient on $K.$
For $z, w \in K$ with $\norm{z - w} < \delta,$ the common neighborhood gives $\norm{\nabla g_P (z) - \nabla g_P (w)} \leq C_0 \norm{z - w}^{1 / 2}.$
For the remaining pairs,
$$
    \norm{\nabla g_P (z) - \nabla g_P (w)} \leq 2 M_0 \leq 2 M_0 \delta^{-1 / 2} \norm{z - w}^{1 / 2}.
$$
Thus the gradient is $1 / 2$-H\"older on each of the two closed sets, proving the lemma.
\end{proof}
We now use this regularity to bound the Taylor remainder of the influence function, including contributions from segments crossing $T_P (\partial A).$
\begin{claim}[Adjoint remainder]\label{clm:subgroup-score-remainder}
Let $g_P$ be the solution of (\ref{eq:subgroupadjoint}) and set $\psi_P = g_P \circ F_P.$
Then $\psi_P$ is globally Lipschitz on $\overline \Lambda,$ and there is a constant $C$ such that, for every probability law $R$ on $\overline \Lambda,$
$$
    \int_\Omega\bigl|\psi_P (T_R) - \psi_P (T_P) - \nabla \psi_P (T_P)^\top (T_R - T_P)\bigr| d \mu
    \leq C \norm{T_R - T_P}_{L^2 (\mu)}^{4 / 3}.
$$
\par
\end{claim}
\begin{proof}[Proof of Claim \ref{clm:subgroup-score-remainder}]
\setlength{\parindent}{24pt}
Lemma \ref{lem:subgroupregularity} and composition with $F_P$ show that $\psi_P = g_P \circ F_P$ is globally Lipschitz and is $C^{1, 1 / 2}$ on each of $\overline{T_P (A)}$ and $\overline{\Lambda \setminus \overline{T_P (A)}}.$
The interface $\Sigma_P \coloneqq T_P (\partial A)$ is a compact $C^{2, \alpha}$ hypersurface contained in $\Lambda.$
Since $P (\Sigma_P) = 0$ and $(T_P)_\# \mu = P,$ the gradient $\nabla \psi_P (T_P (z))$ is defined for $\mu$-almost every $z.$

For any probability law $R$ on $\overline \Lambda,$ write the pointwise Taylor remainder as
$$
    E_{R, P} (z) \coloneqq \psi_P (T_R (z)) - \psi_P (T_P (z)) - \nabla \psi_P (T_P (z))^\top (T_R (z) - T_P (z)).
$$
Let $\mathcal C_{R, P}$ be the set of $z$ for which the segment joining $T_P (z)$ to $T_R (z)$ intersects $\Sigma_P.$
For $z \notin \mathcal C_{R, P},$ this segment stays on one side of the interface.
Taylor's formula and the $1 / 2$-H\"older gradient bound give $|E_{R, P} (z)| \leq C \norm{T_R (z) - T_P (z)}^{3 / 2}.$
For $z \in \mathcal C_{R, P},$ global Lipschitz continuity instead gives $|E_{R, P} (z)| \leq 2 \operatorname{Lip} (\psi_P) \norm{T_R (z) - T_P (z)}.$
Integration and H\"older's inequality yield
\begin{equation} \label{eq:subgrouptaylor}
    \int_\Omega |E_{R, P}|  d \mu \leq C \norm{T_R - T_P}_{L^2 (\mu)}^{3 / 2} + C \int_{\mathcal C_{R, P}} \norm{T_R - T_P}  d \mu.
\end{equation}
All constants are independent of $R.$

It remains to bound the contribution of segments crossing the interface.
For $\delta > 0,$ let $U_\delta \coloneqq \{y \in \Lambda \mid \operatorname{dist} (y, \Sigma_P) \leq \delta\}.$
The normal-coordinate formula for a compact $C^2$ hypersurface gives $|U_\delta| \leq C \delta$ for small $\delta$; see \cite[Lemma 14.16 and formula (14.98)]{gilbarg2001elliptic}.
The density of $P$ is bounded, and increasing $C$ extends the resulting bound to all $\delta > 0,$ since $P (U_\delta) \leq 1.$
Thus
\begin{equation} \label{eq:subgrouptube}
    P (U_\delta) \leq C \delta, \qquad \delta > 0.
\end{equation}
If a segment crosses $\Sigma_P,$ its initial point is within its length of $\Sigma_P.$
Consequently,
$$
    \mathcal C_{R, P} \subset \{z \mid T_P (z) \in U_\delta\} \cup \{z \mid \norm{T_R (z) - T_P (z)} > \delta\}.
$$
Cauchy--Schwarz on the first set and $a \indic_{\{a > \delta\}} \leq a^2 / \delta$ on the second give
\begin{align*}
    \int_{\mathcal C_{R, P}} \norm{T_R - T_P}  d \mu
    & \leq \norm{T_R - T_P}_{L^2 (\mu)} P (U_\delta)^{1 / 2} + \delta^{-1} \norm{T_R - T_P}_{L^2 (\mu)}^2 \\
    & \leq C \delta^{1 / 2} \norm{T_R - T_P}_{L^2 (\mu)} + \delta^{-1} \norm{T_R - T_P}_{L^2 (\mu)}^2.
\end{align*}
The bound is zero if $T_R = T_P$ in $L^2 (\mu).$
Otherwise, choose $\delta = \norm{T_R - T_P}_{L^2 (\mu)}^{2 / 3}$ to obtain
$$
    \int_{\mathcal C_{R, P}} \norm{T_R - T_P}  d \mu \leq C \norm{T_R - T_P}_{L^2 (\mu)}^{4 / 3}.
$$
Finally, $\norm{T_R - T_P}_{L^2 (\mu)} \leq \operatorname{diam} (\Lambda)$ implies
$$
    \norm{T_R - T_P}_{L^2 (\mu)}^{3 / 2} \leq \operatorname{diam} (\Lambda)^{1 / 6} \norm{T_R - T_P}_{L^2 (\mu)}^{4 / 3}.
$$
Substitution into (\ref{eq:subgrouptaylor}) proves the claim.
\end{proof}
\begin{proof}[Proof of Example \ref{ex:subgroup}]
\setlength{\parindent}{24pt}
We verify Assumption \ref{ass:functional} with $r = 4 / 3.$
Since $A$ is nonempty and open and $\rho$ is bounded below, $\mu (A) > 0.$
Set
$$
    c_P (z) \coloneqq \frac{\nabla_y \ell (z, T_P (z))}{\mu (A)}, \qquad D_P (v) \coloneqq \int_A c_P^\top \nabla v  d \mu.
$$
The chain rule and Lemma \ref{lem:mapregularity} give $c_P \in C^{1, \alpha} (\overline \Omega; \bR^d).$
Cauchy--Schwarz gives $|D_P (v)| \leq \norm{\nabla_y \ell}_\infty \mu (A)^{-1 / 2} \norm{v}_\cV,$ so $D_P \in \cV^\ast.$

For any law $R$ on $\overline \Lambda,$ convexity of $\Lambda$ and Taylor's formula in the second argument of $\ell$ give
\begin{align*}
    |\cJ_{A, \ell} (\varphi_R) - \cJ_{A, \ell} (\varphi_P) - D_P (\varphi_R - \varphi_P)|
    & \leq \frac{\norm{D_y^2 \ell}_\infty}{2 \mu (A)} \norm{T_R - T_P}_{L^2 (\mu)}^2 \\
    & \leq C \norm{T_R - T_P}_{L^2 (\mu)}^{4 / 3}.
\end{align*}
The second inequality uses $\norm{T_R - T_P}_{L^2 (\mu)} \leq \operatorname{diam} (\Lambda).$
This proves condition (2).

The adjoint of $D_P$ solves (\ref{eq:subgroupadjoint}) with the coefficient $c_P$ above.
Lemma \ref{lem:subgroupregularity} makes it continuous, proving condition (1), and Claim \ref{clm:subgroup-score-remainder} gives condition (3) with $r = 4 / 3.$
\end{proof}
\begin{proof}[Proof of Example \ref{ex:tail}]
\setlength{\parindent}{24pt}
The functional is linear, with first variation $D_\tau (v) = \tau^{-1} \int_{E_\tau} z^\top \nabla v  d \mu.$
Since $\norm{z} \leq 1$ and $\mu (E_\tau) = \tau,$ Cauchy--Schwarz gives $|D_\tau (v)| \leq \tau^{-1 / 2} \norm{v}_\cV.$
Thus condition (2) holds with zero remainder.

To verify conditions (1) and (3), write
$$
    \cJ_0 (\varphi) \coloneqq \int_\Omega z^\top \nabla \varphi  d \mu, \qquad
    \cJ_{A_\tau} (\varphi) \coloneqq \frac1{1 - \tau} \int_{A_\tau} z^\top \nabla \varphi  d \mu.
$$
Then $\cJ_\tau = (\cJ_0 - (1 - \tau) \cJ_{A_\tau}) / \tau.$
Examples \ref{ex:linear} and \ref{ex:subgroup} apply with $w (z) = z$ and $\ell (z, y) = z^\top y,$ respectively.
Their adjoints are continuous, and their influence functions, denoted by $\psi_{0, P}$ and $\psi_{A_\tau, P},$ satisfy (\ref{eq:psibound}) with exponents $2$ and $4 / 3.$
By linearity and uniqueness of the adjoint,
$$
    \psi_{\tau, P} = \frac1\tau \left(\psi_{0, P} - (1 - \tau) \psi_{A_\tau, P}\right).
$$
The same combination of the continuous adjoints proves condition (1).
The triangle inequality bounds the remainder in (\ref{eq:psibound}) by
$$
    \frac1\tau \left(C_{0, P} \norm{T_R - T_P}_{L^2 (\mu)}^2 + (1 - \tau) C_{A_\tau, P} \norm{T_R - T_P}_{L^2 (\mu)}^{4 / 3}\right) \leq C \norm{T_R - T_P}_{L^2 (\mu)}^{4 / 3},
$$
since $\norm{T_R - T_P}_{L^2 (\mu)} \leq \operatorname{diam} (\Lambda).$
This proves condition (3) and completes the verification.
\end{proof}

\subsection{Empirical transport bounds and inference} \label{app:empiricalrate}
\begin{proof}[Proof of Lemma \ref{lem:empiricalrate}]
\setlength{\parindent}{24pt}
Assumption \ref{ass:sampling} gives i.i.d. observations with density $p.$
Under Assumption \ref{ass:base}, $\Lambda$ is bounded and convex with $C^{2, 1}$ boundary, and $p$ is H\"older continuous and bounded above and away from zero.
Consequently, \cite[Theorem 1.1]{ambrosio2022quadratic} for $d = 2$ and \cite[Theorem 1.1(ii)]{huesmann2025asymptotics} with cost exponent $2$ for $d \geq 3$ give $\bE [W_2^2 (P_n, P)] \leq C a_{n, d}$ for all sufficiently large $n.$
Increasing $C$ covers the remaining $n \geq 2,$ since $W_2^2 (P_n, P) \leq \operatorname{diam} (\Lambda)^2.$
Lemma \ref{lem:stability} now yields
$$
    \bE \left[\norm{T_{P_n} - T_P}_{L^2 (\mu)}^2 \right] \leq C_P \bE [W_2^2 (P_n, P)] \leq C a_{n, d}
    ,
$$
which completes the proof.
\end{proof}
\begin{proof}[Proof of Proposition \ref{prop:empiricalbound}]
\setlength{\parindent}{24pt}
Taking $R = P_n$ in Theorem \ref{thm:expansion} gives (\ref{eq:empiricalexp}).
Since $r \leq 2,$ Jensen's inequality and Lemma \ref{lem:empiricalrate} give
$$
    \bE [| \cR (P_n, P) |]
    \leq
    C \bE \left[\norm{T_{P_n} - T_P}_{L^2 (\mu)}^r\right]
    \leq
    C \left( \bE \left[\norm{T_{P_n} - T_P}_{L^2 (\mu)}^2\right] \right)^{r / 2}
    \leq
    C a_{n, d}^{r / 2}
    ,
$$
which completes the proof.
\end{proof}
\begin{proof}[Proof of Theorem \ref{thm:empirical}]
\setlength{\parindent}{24pt}
Proposition \ref{prop:empiricalbound} and condition (\ref{eq:dimensionthreshold}) give, by Markov's inequality,
$$
    \sqrt n \cR (P_n, P) = o_P (1).
$$
Thus (\ref{eq:empiricalexp}) is an asymptotically linear representation.

The influence function is bounded and centered: $g_P$ is continuous on the compact set $\overline \Omega,$ $F_P$ is continuous, and $P \psi_P = 0.$
The central limit theorem and Slutsky's theorem therefore give the asserted Gaussian limit.
The identity $(F_P)_\# P = \mu$ gives $\sigma_P^2 = P \psi_P^2 = \mu g_P^2.$
\end{proof}

\begin{proof}[Proof of Corollary \ref{cor:joint}]
\setlength{\parindent}{24pt}
Write $\boldsymbol \theta \coloneqq (\theta_j (P))_{j = 1}^m,$ $\widehat{\boldsymbol \theta}_n \coloneqq (\widehat \theta_{n, j})_{j = 1}^m,$ and $\boldsymbol \psi_P \coloneqq (\psi_{P, j})_{j = 1}^m.$
Since $m$ is fixed, the componentwise remainder bounds in Proposition \ref{prop:empiricalbound} give
$$
    \sqrt n (\widehat{\boldsymbol \theta}_n - \boldsymbol \theta) = \frac{1}{\sqrt n} \sum_{i = 1}^n \boldsymbol \psi_P (Y_i) + o_P (1).
$$
The vectors $\boldsymbol \psi_P (Y_i)$ are i.i.d., centered, and bounded, with covariance $\Sigma,$ so the multivariate central limit theorem proves the corollary, including singular $\Sigma.$
\end{proof}

\begin{proof}[Proof of Theorem \ref{thm:bootstrap}]
\setlength{\parindent}{24pt}
The law $P$ has finite second moment because $\overline \Lambda$ is compact, the observations are i.i.d. with law $P,$ and $P_n^{\ast}$ is their empirical bootstrap measure.
Thus the Wasserstein inequality in \cite[Lemma 2.9]{gonzalez2026empirical} applies; together with Lemma \ref{lem:empiricalrate}, it gives
$$
    \bE [W_2^2 (P_n^{\ast}, P)]
    \leq 4 \bE [W_2^2 (P_n, P)]
    \leq C a_{n, d}.
$$
Set $L_n^{\ast} \coloneqq \sqrt n (P_n^{\ast} - P_n) \psi_P.$
We first prove $\bE [|S_n^{\ast} - L_n^{\ast}|] = o (1).$
Applying Theorem \ref{thm:expansion} at $P_n^{\ast}$ and $P_n$ yields
$$
    S_n^{\ast} = L_n^{\ast} + \sqrt n \left(\cR (P_n^{\ast}, P) - \cR (P_n, P) \right).
$$
The remainder bound in that theorem, Lemma \ref{lem:stability}, and Jensen's inequality with $r \leq 2$ now give
$$
    \bE {[}|S_n^{\ast} - L_n^{\ast}|{]}
    \leq \sqrt n \bE \left[| \cR (P_n^{\ast}, P)| + | \cR (P_n, P)| \right]
    \leq C \sqrt n a_{n, d}^{r / 2} = o (1),
$$
where the last equality follows from (\ref{eq:dimensionthreshold}).

The variables $\psi_P (Y_i)$ are i.i.d., centered, and bounded, with variance $\sigma_P^2,$ and $\psi_P (Y_i^{\ast})$ are an empirical bootstrap resample of these variables.
The bootstrap theorem for sample means, \cite[Theorem 23.4]{van1998asymptotic}, therefore gives conditional weak convergence of $L_n^{\ast}$ to $N (0, \sigma_P^2)$ almost surely, including when $\sigma_P^2 = 0.$

Moreover,
$$
    \sup_{f \in \mathrm{BL}_1 (\bR)}
    | \bE^{\ast} {[}f (S_n^{\ast}){]} - \bE^{\ast} {[}f (L_n^{\ast}){]}|
    \leq \bE^{\ast} {[}|S_n^{\ast} - L_n^{\ast}|{]} = o_P (1)
$$
by the preceding expectation bound and Markov's inequality.
Since the bounded Lipschitz metric metrizes weak convergence, this proves the first assertion.

For the coverage statement, suppose $\sigma_P^2 > 0.$
The limiting normal distribution function is then continuous.
Theorem \ref{thm:empirical} and the conditional limit just proved therefore verify the hypotheses of \cite[Lemma 23.3]{van1998asymptotic}, with scale $n^{- 1 / 2}$ and equal tail probabilities $\gamma / 2.$
It gives the asserted coverage of (\ref{eq:CI}).
\end{proof}

\subsection{Failure of process convergence}\label{app:processproofs}
\begin{proof}[Proof of Theorem \ref{thm:processobstruction}]
\setlength{\parindent}{24pt}
Set $U_n \coloneqq \sqrt n (\varphi_{P_n} - \varphi_P).$
We first prove that no subsequence of $U_n$ converges weakly in $H^1 (\Omega).$
Suppose, to the contrary, that $U_n \rightsquigarrow U$ in $H^1 (\Omega)$ along a subsequence.
Since $\cV$ is closed, $U \in \cV$ almost surely and $\norm{U_n}_\cV = O_P (1).$

We identify the law of the scalar projections $B_P (U, w),$ defined in equation (\ref{eq:bilinear}), first for smooth test functions $w.$
Let $\cD$ be the set of functions in $\cV$ that extend smoothly to a neighborhood of $\overline \Omega.$
Smooth approximation on the Lipschitz domain $\Omega,$ followed by subtraction of the $\mu$-mean, shows that $\cD$ is dense in $\cV.$

For $w \in \cD,$ set $f_w \coloneqq w \circ F_P,$ so that $f_w \circ T_P = w.$
Lemma \ref{lem:mapregularity} gives $f_w \in C^2 (\overline \Lambda)$ and $\nabla f_w (T_P) = H_P^{- 1} \nabla w.$
The push-forward identities and Taylor's formula on the convex domain $\Lambda$ give
\begin{equation} \label{eq:obstructiontaylor}
    (P_n - P) f_w
    =
    B_P (\varphi_{P_n} - \varphi_P, w)
    +
    R_{n, w}
    ,
    \qquad
    | R_{n, w} |
    \leq
    C_w \norm{\varphi_{P_n} - \varphi_P}_\cV^2
    ,
\end{equation}
with $C_w$ independent of $n.$

Multiplying (\ref{eq:obstructiontaylor}) by $\sqrt n,$ the bound $\norm{U_n}_\cV = O_P (1)$ gives $\sqrt n R_{n, w} = O_P (n^{- 1 / 2}) = o_P (1).$
Since $f_w (Y_i) = w (F_P (Y_i)),$ $F_P (Y_i)$ has law $\mu,$ and $\mu w = 0,$ we obtain
\begin{equation} \label{eq:obstructionprojection}
    B_P (U_n, w)
    =
    \frac{1}{\sqrt n}
    \sum_{i = 1}^n
        w (F_P (Y_i))
    +
    o_P (1)
    .
\end{equation}
For any $m \geq 1$ and $w_1, \ldots, w_m \in \cD,$ (\ref{eq:obstructionprojection}) and the multivariate central limit theorem give
$$
    \left(B_P (U_n, w_j)\right)_{j = 1}^m
    \rightsquigarrow N (0, \Sigma),
    \qquad
    \Sigma_{jk} \coloneqq \int_\Omega w_j w_k d \mu.
$$
The continuous mapping theorem, applied to $U_n \rightsquigarrow U,$ also gives convergence of this vector to $\left(B_P (U, w_j)\right)_{j = 1}^m.$
Uniqueness of the limit in distribution therefore shows that $\left(B_P (U, w_j)\right)_{j = 1}^m$ has law $N (0, \Sigma).$

We now prove the same assertion for arbitrary $w_1, \ldots, w_m \in \cV.$
Choose $w_{j, r} \in \cD$ such that $w_{j, r} \to w_j$ in $\cV$ for each $j.$
Boundedness of $B_P$ and $U \in \cV$ give, almost surely,
$$
    |B_P (U, w_{j, r}) - B_P (U, w_j)|
    \leq C \norm{U}_\cV \norm{w_{j, r} - w_j}_\cV
    \longrightarrow 0.
$$
Thus the vectors $\left(B_P (U, w_{j, r})\right)_{j = 1}^m$ converge almost surely to $\left(B_P (U, w_j)\right)_{j = 1}^m.$
By the preceding paragraph, their laws are $N (0, \Sigma^{(r)}),$ where
$$
    \Sigma^{(r)}_{jk}
    = \int_\Omega w_{j, r} w_{k, r} d \mu
    \longrightarrow \int_\Omega w_j w_k d \mu
    = \Sigma_{jk},
$$
since convergence in $\cV$ implies convergence in $L^2 (\mu).$
These Gaussian laws converge weakly to $N (0, \Sigma),$ so uniqueness of the limit in distribution gives the asserted law for $\left(B_P (U, w_j)\right)_{j = 1}^m.$

We now show that this Gaussian law is incompatible with $U \in \cV,$ following the spectral argument of \cite[Section 5.3]{manole2024central}.
Boundedness and coercivity of $B_P,$ together with Poincar\'e's inequality, make $u \mapsto B_P (u, u)^{1 / 2}$ a norm equivalent to the $H^1$ norm on $\cV.$
The positive upper and lower bounds on $\rho$ allow us to apply the Rellich compactness theorem and smooth approximation: $\cV$ embeds compactly and densely into $L_0^2 (\mu).$
Thus the spectral theorem for compactly embedded coercive forms (\cite[Theorem 4.1]{laugesen2012spectral}) gives an orthonormal basis $(e_j)$ of $L_0^2 (\mu),$ with $e_j \in \cV,$ and eigenvalues $0 < \lambda_1 \leq \lambda_2 \leq \cdots \to \infty$ satisfying
\begin{equation} \label{eq:energyspectrum}
    B_P (e_j, w)
    =
    \lambda_j \int_\Omega e_j w d \mu
    ,
    \qquad
    w \in \cV
    .
\end{equation}
We next prove $\lambda_j \leq C j^{2 / d},$ which will imply $\sum_j \lambda_j^{- 1} = \infty$ when $d \geq 2.$
The min--max formula (\cite[equation (9.2)]{laugesen2012spectral}), with the minimum taken over $j$-dimensional linear subspaces of $\cV,$ states that
$$
    \lambda_j = \min_{\substack{S \subset \cV \\ \dim S = j}}
    \max_{v \in S \setminus \{0\}}
    \frac{B_P (v, v)}{\int_\Omega v^2 d \mu}.
$$
To construct a suitable subspace $S,$ let $0 < \kappa_1 \leq \kappa_2 \leq \cdots$ be the eigenvalues of the Dirichlet Laplacian on $\Omega,$ with eigenfunctions $u_k$ orthonormal for Lebesgue measure.
Thus $- \Delta u_k = \kappa_k u_k$ in $\Omega$ and $u_k = 0$ on $\partial \Omega.$
Weyl's law on the bounded smooth domain $\Omega$ (\cite[Theorem 11.1]{laugesen2012spectral}) gives $\kappa_k \leq C k^{2 / d}.$

For $v = \sum_{k = 1}^{j + 1} a_k u_k \neq 0,$ the bounds on $\rho$ and $H_P^{- 1},$ followed by integration by parts, give
$$
    \frac{B_P (v, v)}{\int_\Omega v^2 d \mu}
    \leq C \frac{\int_\Omega \norm{\nabla v}^2 d z}{\int_\Omega v^2 d z}
    = C \frac{\sum_{k = 1}^{j + 1} \kappa_k a_k^2}{\sum_{k = 1}^{j + 1} a_k^2}
    \leq C \kappa_{j + 1},
$$
with $C$ independent of $j.$
The constraint $\mu v = 0$ removes at most one dimension from $\operatorname{span} (u_1, \ldots, u_{j + 1}).$
Its mean-zero subspace therefore contains a $j$-dimensional subspace $S \subset \cV.$
Applying the min--max formula to this $S$ yields $\lambda_j \leq C \kappa_{j + 1} \leq C^{\prime} j^{2 / d}.$

We now obtain the contradiction.
The joint Gaussian law proved above and orthonormality of $(e_j)$ in $L_0^2 (\mu)$ imply that $Z_j \coloneqq B_P (U, e_j)$ are independent standard normal variables.
Equation (\ref{eq:energyspectrum}) makes $\left(e_j / \sqrt{\lambda_j}\right)$ orthonormal for the inner product $B_P,$ so Bessel's inequality gives
$$
    \sum_{j = 1}^k \frac{Z_j^2}{\lambda_j} \leq B_P (U, U) \leq C \norm{U}_\cV^2 < \infty
    \qquad \text{almost surely, for every } k.
$$
For a standard normal variable, $\bE [\exp (- Z_j^2 / \lambda_j)] = (1 + 2 / \lambda_j)^{- 1 / 2}.$
Exponentiating the preceding inequality and using independence therefore gives
\par
$$
    0
    <
    \bE [\exp ( - B_P (U, U) )]
    \leq
    \prod_{j = 1}^k
        (1 + 2 / \lambda_j)^{- 1 / 2}
    \longrightarrow
    0
    .
$$
Indeed, $\sum_j \lambda_j^{- 1} \geq C^{- 1} \sum_j j^{- 2 / d} = \infty$ for $d \geq 2,$ and $\log (1 + 2 / \lambda_j) \sim 2 / \lambda_j,$ so the product tends to zero.
The expectation on the left is strictly positive because $B_P (U, U)$ is finite almost surely.
This contradiction excludes every weakly convergent subsequence of $U_n.$
Prokhorov's theorem on the separable Hilbert space $H^1 (\Omega)$ then excludes tightness.

Suppose $W_n \coloneqq a_n (\varphi_{P_n} - \varphi_P)$ is tight in $H^1 (\Omega),$ with $a_n \to \infty.$
We prove that $a_n = o \left(\sqrt n\right).$
Otherwise, for some $\varepsilon > 0,$ there is a subsequence with $a_n \geq \varepsilon \sqrt n.$
Passing to a further subsequence gives $\sqrt n / a_n \to b \in [0, \varepsilon^{- 1}].$
Tightness and Prokhorov's theorem give a further subsequence with $W_n {\rightsquigarrow} W.$
Then $U_n = \left(\sqrt n / a_n\right)W_n \rightsquigarrow bW,$ contradicting the first part.
Hence $a_n = o \left(\sqrt n\right).$

We next prove that every subsequential weak limit of $W_n$ is zero.
Take a subsequence such that $W_n {\rightsquigarrow} W.$
Its limit lies in the closed space $\cV,$ and tightness implies $\norm{W_n}_\cV = O_P (1).$
For each $w \in \cD,$ the scalar central limit theorem gives $(P_n - P) f_w = O_P (n^{- 1 / 2}).$
Multiplying (\ref{eq:obstructiontaylor}) by $a_n$ and using $a_n |R_{n, w}| \leq C_w a_n^{- 1} \norm{W_n}_\cV^2,$ we obtain
$$
    B_P (W_n, w)
    = a_n (P_n - P) f_w + O_P (a_n^{- 1})
    = O_P \left(a_n / \sqrt n\right) + O_P (a_n^{- 1})
    = o_P (1).
$$
The continuous mapping theorem yields $B_P (W, w) = 0$ almost surely.
Choose a countable subset of $\cD$ dense in $\cV$ and intersect these probability-one events.
Continuity of $B_P (W, \cdot)$ then gives $B_P (W, w) = 0$ for every $w \in \cV$ on that event; taking $w = W$ and using coercivity proves $W = 0.$

Every subsequence of the tight sequence therefore has a further subsequence converging in distribution to zero.
Thus $W_n {\rightsquigarrow} 0,$ which for a constant limit is equivalent to $\norm{W_n}_{H^1 (\Omega)} \to 0$ in probability.

For the assertion about transport maps, $u \mapsto \nabla u$ is an isometry from $\cV$ onto a closed subspace of $L^2 (\mu; \bR^d).$
Poincar\'e's inequality makes its inverse continuous into $H^1 (\Omega).$
Tightness of the normalized map errors therefore gives tightness of the corresponding normalized potentials; applying the result above and the continuous gradient map proves the assertion for transport maps.
\end{proof}

\subsection{Efficiency proofs} \label{app:efficiency}
\begin{proof}[Proof of Theorem \ref{thm:efficiency}]
\setlength{\parindent}{24pt}
(i)
Fix $\eta \in \cH_P.$
Since $p$ is bounded away from zero, \cite[Lemma 2.2]{ambrosio2022quadratic} gives, for all sufficiently small $t,$
$$
\begin{aligned}
    W_2^2 (P_{t, \eta}, P)
    \leq 
    C t^2
    \sup_{\substack{v \in H^1 (\Lambda) \\ \norm{\nabla v}_{L^2 (\Lambda)} \leq 1}}
    \left|\int_\Lambda v \eta p  d y\right|^2
    \leq 
    C t^2 \norm{\eta p}_{L^2 (\Lambda)}^2
    \leq 
    C t^2 P \eta^2.
\end{aligned}
$$
For the second inequality, subtract the Lebesgue mean of $v$ using $\int_\Lambda \eta p  d y = 0,$ and apply Poincar\'e's inequality; the last inequality uses boundedness of $p.$
Lemma \ref{lem:stability} then gives $\norm{T_{P_{t, \eta}} - T_P}_{L^2 (\mu)} = O (t).$
Theorem \ref{thm:expansion} yields
$$
    \theta (P_{t, \eta}) - \theta (P)
    = t P (\eta \psi_P) + O (t^r).
$$
Since $r > 1,$ the derivative is $\dot\theta_P (\eta) = P (\eta \psi_P),$ a continuous linear form on $\cH_P$ by Cauchy--Schwarz.
Moreover, $\psi_P \in \cH_P,$ so
$$
    \sup_{\substack{\eta \in \cH_P \\ P \eta^2 > 0}}
    \frac{|\dot\theta_P (\eta)|^2}{P \eta^2}
    = P \psi_P^2 = \sigma_P^2.
$$
This proves (i).
In particular, setting $t = n^{- 1 / 2}$ gives
\begin{equation} \label{eq:parameterderivative}
    \sqrt n \left(\theta (P_{n, \eta}) - \theta (P)\right)
    = P (\eta \psi_P) + o (1).
\end{equation}

(ii)
For each $\eta \in \cH_P,$ boundedness of $\eta$ makes the path (\ref{eq:scorepath}) differentiable in quadratic mean at zero with score $\eta.$
Hence \cite[Lemma 3.10.11]{van1996weak} gives local asymptotic normality: under $P^{\otimes n},$
$$
    \log \frac{d P_{n, \eta}^{\otimes n}}{d P^{\otimes n}}
    = \frac1{\sqrt n} \sum_{i = 1}^n \eta (Y_i) - \frac12 P \eta^2 + o_P (1).
$$
The multivariate central limit theorem gives joint convergence of the centered sums for every finite collection of scores; the limiting covariance for $\eta, \zeta \in \cH_P$ is $P (\eta \zeta).$

Together with (\ref{eq:parameterderivative}) and the continuous linear derivative identified in (i), this verifies the hypotheses of \cite[Theorem 3.11.2]{van1996weak} and gives the convolution bound in (ii).

(iii)
The same local asymptotic normality and parameter regularity allow us to apply \cite[Theorem 3.11.5]{van1996weak}.
Its lower semicontinuous subconvex losses are the bowl-shaped losses defined in Section \ref{sec:efficiency}, and its Gaussian bound has variance $\sigma_P^2$ by the calculation above.
This proves (iii).

(iv)
Under the additional assumptions in (iv), Theorem \ref{thm:empirical} gives
$$
    \sqrt n (\widehat \theta_n - \theta (P))
    = \frac{1}{\sqrt n} \sum_{i = 1}^n \psi_P (Y_i) + o_P (1).
$$
For each $\eta \in \cH_P,$ the bivariate central limit theorem and the likelihood expansion give, under $P^{\otimes n},$
$$
    \begin{pmatrix}
        \sqrt n (\widehat \theta_n - \theta (P)) \\
        \displaystyle \log \frac{d P_{n, \eta}^{\otimes n}}{d P^{\otimes n}}
    \end{pmatrix}
    \rightsquigarrow
    N \left(
        \begin{pmatrix}
            0 \\
            - \frac12 P \eta^2
        \end{pmatrix},
        \begin{pmatrix}
            \sigma_P^2 & P (\eta \psi_P) \\
            P (\eta \psi_P) & P \eta^2
        \end{pmatrix}
    \right).
$$
The exponential of the second Gaussian coordinate has expectation one, since its mean is minus half its variance.
Le Cam's third lemma (\cite[Theorem 3.10.7]{van1996weak}) therefore identifies the limit under $P_{n, \eta}^{\otimes n}$ as the law of the first coordinate after exponential tilting by the second.
This tilting shifts the mean of the first coordinate by their covariance $P (\eta \psi_P)$ and leaves its variance $\sigma_P^2$ unchanged.
Thus
$$
    \sqrt n (\widehat \theta_n - \theta (P))
    \rightsquigarrow N \left(P (\eta \psi_P), \sigma_P^2\right)
    \quad \text{under } P_{n, \eta}^{\otimes n}.
$$
Subtracting (\ref{eq:parameterderivative}) yields
$$
    \sqrt n \left(\widehat \theta_n - \theta (P_{n, \eta})\right)
    \rightsquigarrow N (0, \sigma_P^2)
    \quad \text{under } P_{n, \eta}^{\otimes n}.
$$
The limit is independent of $\eta,$ proving regularity.

Under the assumptions in (iv), estimator regularity implies convergence of expectations to $\bE [\ell (Z)]$ for every bounded continuous loss $\ell.$
Since $I$ is finite, this convergence is uniform over $\eta \in I,$ proving the remaining assertion in (iv).
\end{proof}
\setlength{\emergencystretch}{2em}
\printbibliography
\end{document}